\documentclass[aop,seceqn,citesort,MSNbibl,dvips]{arximspdf}

\doi{10.1214/10-AOP607}
\volume{39}
\issue{6}
\pubyear{2011}
\firstpage{2271}
\lastpage{2317}

\makeatletter

\newtheorem{theorem}{Theorem}[section]
\newtheorem{lemma}[theorem]{Lemma}
\newtheorem{corollary}[theorem]{Corollary}
\newtheorem{proposition}[theorem]{Proposition}

\newproclaim{assumption}{Assumption}[section]
\newproclaim{remark}{Remark}[section]

\newcommand{\const}{\operatorname{const}}
\newcommand{\tr}{\operatorname{tr}}
\newcommand{\supp}{\operatorname{supp}}

\mathchardef\varPi="0105
\mathchardef\varGamma="0100
\mathchardef\varPhi="0108
\mathchardef\varPsi="0109

\newcommand{\PP}{P}
\newcommand{\QQ}{Q}
\newcommand{\EE}{E}
\newcommand{\Var}{\operatorname{Var}}
\newcommand{\Cov}{\operatorname{Cov}}
\newcommand{\sfP}{\mathsf{P}}
\newcommand{\sfE}{\mathsf{E}}
\newcommand{\NN}{\mathbb{N}}
\newcommand{\ZZ}{\mathbb{Z}}
\newcommand{\RR}{\mathbb{R}}
\newcommand{\CC}{{\mathbb C}}
\newcommand{\CP}{{\varPi}}
\newcommand{\calX}{\mathcal{X}}

\newcommand{\e}{{ e}}
\newcommand{\dif}{{ d}}
\newcommand{\topp}{\top}

\makeatother

\begin{document}
\begin{frontmatter}

\title{Universality of the limit shape of convex lattice polygonal lines}
\runtitle{Universality of the limit shape}

\begin{aug}
\author[A]{\fnms{Leonid V.} \snm{Bogachev}\corref{}\thanksref{t1}\ead[label=e1]{L.V.Bogachev@leeds.ac.uk}} and
\author[B]{\fnms{Sakhavat M.} \snm{Zarbaliev}\thanksref{t2}\ead[label=e2]{szarbaliev@mail.ru}}
\runauthor{L. V. Bogachev and S. M. Zarbaliev}
\affiliation{University of Leeds and International Institute of
Earthquake Prediction Theory and Mathematical Geophysics}
\address[A]{Department of Statistics\\
School of Mathematics\\
University of Leeds\\
Leeds LS2 9JT\\
United Kiingdom\\
\printead{e1}}
\address[B]{International Institute\\
\quad of Earthquake Prediction Theory\\
\quad and Mathematical Geophysics\\
Moscow 117997\\
Russia\\
\printead{e2}}
\end{aug}

\thankstext{t1}{Supported in part by a Leverhulme
Research Fellowship.}

\thankstext{t2}{Supported in part by DFG Grant 436 RUS 113/722.}

\received{\smonth{7} \syear{2008}}
\revised{\smonth{2} \syear{2010}}

%
\begin{abstract}
Let $\CP_n$ be the set of convex polygonal lines $\varGamma$ with
vertices on~$\ZZ_+^2$ and fixed endpoints $0=(0,0)$ and
$n=(n_1,n_2)$. We are concerned with the \textit{limit shape}, as
$n\to\infty$, of ``typical'' $\varGamma\in\CP_n$ with respect to
a~parametric family of probability measures $\{\PP_n^r, 0<r<\infty\}$
on~$\CP_n$, including the uniform distribution ($r=1$) for which the
limit shape was found in the early 1990s independently by
A. M. Vershik, I. B\'{a}r\'{a}ny and Ya. G. Sinai. We show that, in
fact, the limit shape is \textit{universal} in the class
$\{\PP^r_n\}$, even though $\PP^r_n$ ($r\ne1$) and $\PP^1_n$ are
asymptotically singular. Measures $\PP^r_n$ are constructed,
following Sinai's approach, as conditional distributions
$\QQ_z^r(\cdot | \CP_n)$, where $\QQ
_z^r$ are suitable
product measures on the space $\CP=\bigcup_n\CP_n$, depending on an
auxiliary ``free'' parameter $z=(z_1,z_2)$. The transition from
$(\CP,\QQ_z^r)$ to $(\CP_n,\PP_n^r)$ is based on the asymptotics of
the probability $\QQ_z^r(\CP_n)$, furnished by a certain
two-dimensional local limit theorem. The proofs involve subtle
analytical tools including the M\"{o}bius inversion formula and
properties of zeroes of the Riemann zeta function.
\end{abstract}

%

\end{frontmatter}

\section{Introduction}\label{sec1}

\subsection{Background: The limit shape}\label{sec1.1}

In this paper, a \textit{convex lattice polygonal line} $\varGamma$ is
a piecewise linear path on the plane, starting at the origin
$0=(0,0)$, with vertices on the integer lattice
$\ZZ^2_+:=\{(i,j)\in\ZZ^2\dvtx i,j\ge0\}$, and such that the inclination
of its consecutive edges strictly increases staying between $0$ and
$\pi/2$. Let $\CP$ be the set of all convex lattice polygonal lines
with finitely many edges, and denote by $\CP_{n}\subset\CP$ the
subset of polygonal lines $\varGamma\in\CP$ whose right endpoint
$\xi=\xi_\varGamma$ is fixed at $n=(n_1,n_2)\in\ZZ^2_+ $.

We are concerned with the problem of \textit{limit shape} of
``typical'' $\varGamma\in\CP_n$, as $n\to\infty$, with respect to
some probability measure $\PP_n$ on $\CP_n$. Here the ``limit
shape'' is understood as a planar curve $\gamma^*$ such that, with
overwhelming $\PP_n$-probability for large enough $n$, properly
scaled polygonal lines $\tilde\varGamma_n=S_n(\varGamma)$ lie within
an arbitrarily small neighborhood of $\gamma^*$. More precisely, for
any $\varepsilon>0$ it should hold that
%
\begin{equation}\label{eq:LLN}
\lim_{n\to\infty}\PP_n\{d(\tilde\varGamma_n,\gamma^*)\le
\varepsilon\}=1,
\end{equation}
where $d(\cdot,\cdot)$ is some metric on the path space---for
instance, induced by the Hausdorff distance between compact sets (in
$\RR^2$),
%
\begin{equation}\label{eq:dH}
d_{\mathcal H}(A,B):=\max\Bigl\{\max_{x\in A}\min_{y\in B}|x-y|,
 \max_{y\in B}\min_{x\in A}|x-y|\Bigr\},
\end{equation}
where \mbox{$| \cdot |$} is the Euclidean vector norm.

Of course, the limit shape and its very existence may depend on the
pro\-bability law $\PP_n$. With respect to the \textit{uniform
distribution} on $\CP_n$, the prob\-lem was solved independently by
Vershik~\cite{V1}, B\'{a}r\'{a}ny \cite{B} and Sinai~\cite{S}, who
showed that, under the scaling $S_n\dvtx(x_1,x_2)\mapsto
(x_1/n_1,x_2/n_2)$, the limit shape $\gamma^*$ is given by a
parabola arc defined by the Cartesian equation
%
\begin{equation}\label{eq:gamma0}
\sqrt{1-x_1}+\sqrt{x_2}=1,\qquad 0\le x_1,x_2\le1.
\end{equation}
More precisely [cf. (\ref{eq:LLN})], if $n=(n_1,n_2)\to\infty$ so
that $n_2/n_1\to c \in(0,\infty)$ then, for any $\varepsilon>0$,
%
\begin{equation}\label{eq:VBS}
\lim_{n\to\infty}\frac{\#\{\varGamma\in\CP_n\dvtx  d_{\mathcal
H}(\tilde\varGamma_n,\gamma^*)\le\varepsilon\}}{\#(\CP_n)}=1.
\end{equation}
[Here and in what follows, $\#(\cdot)$ denotes the number of
elements in a set.]

The proofs in papers \cite{V1,B} involved a blend of combinatorial,
variational and geometric arguments and were based on a direct
analysis of the corresponding generating function via a multivariate
saddle-point method for a~5Cauchy integral \cite{V1} or a suitable
Tauberian theorem \cite{B}. Extending some of these ideas and using
large deviations techniques, Vershik and~Zeitou\-ni~\cite{VZ}
developed a systematic approach to the limit shape problem for the
uniform measure on more general ensembles of convex lattice
polygonal lines with various geometric restrictions.

Sinai \cite{S} proposed an alternative, probabilistic method
essentially based on randomization of the right endpoint of the
polygonal line $\varGamma\in\CP_n$; we will comment
more on this
approach in Section \ref{sec1.3}. Let us point out that the paper
\cite{S} contained the basic ideas but only sketches of the proofs.
Some of these techniques were subsequently elaborated by Bogachev
and Zarbaliev \cite{BZ1,BZ2} and also by Zarbaliev in his Ph.D. thesis
\cite{Z}; however, a complete proof has not been published as
yet.
\begin{remark}\label{rm:strict}
$\!\!\!$A polygonal line $\varGamma\in\CP_n$ can be viewed as a vector sum~of
its consecutive edges, resulting in a given integer vector
$n=(n_1,n_2)$; due to the convexity property, the order of parts in
the sum is uniquely determined. Hence, any such $\varGamma$
represents an integer \textit{vector partition} of $n\in\ZZ_+^2$ or,
more precisely, a~\textit{strict} vector partition (i.e., without
proportional parts; see \cite{V1}). This observation incorporates
the topic of convex lattice polygonal lines in a~general theory of
integer partitions. For ordinary, one-dimensional partitions, the
problem of limit shape can also be set up, but for a special
geometric object associated with partitions, called \textit{Young
diagrams} \cite{V2,V3}.
\end{remark}

\subsection{Main result}\label{sec1.2}
Vershik \cite{V1}, page 20, pointed out that it would be interesting
to study asymptotic properties of convex lattice polygonal lines
under other probability measures $\PP_n$ on $\CP_n$, and conjectured
that the limit shape might be universal for some classes of
measures. Independently, a~similar hypothesis was put forward by
Prokhorov \cite{Prokhorov}.

In the present paper, we prove Vershik--Prokhorov's universality
conjecture for a parametric family of probability measures
$\PP_n^{r}$ ($0<r<\infty$) on $\CP_n$ defined by
%
\begin{equation}\label{eq:P-r}
\PP_n^r (\varGamma):=\frac{b^{r}(\varGamma)}{  B_n^r}
,\qquad \varGamma\in\CP_n,
\end{equation}
with
%
\begin{equation}\label{eq:b-Gamma}
b^r(\varGamma):=\prod_{e_i\in\varGamma} b_{k_i}^{r},\qquad
B_n^r:=\sum_{\varGamma\in\CP_n}
b^{r}(\varGamma),
\end{equation}
where the product is taken over all edges $e_i$ of
$\varGamma\in\CP_n$, $k_i$ is the number of lattice points on the
edge $e_i$ except its left endpoint and
%
\begin{equation}\label{eq:b-k-r}\quad
b_k^{r}:=\pmatrix{r+k-1\cr k}=\frac{r(r+1)\cdots(r+k-1)}{k!},\qquad
k=0,1,2,\ldots.
\end{equation}
Note that for  $r=1$ the measure (\ref{eq:P-r}) is reduced to the
uniform distribution on~$\CP_n$. Qualitatively, formulas
(\ref{eq:b-Gamma}), (\ref{eq:b-k-r}) introduce certain probability
weights for random edges on $\varGamma$ by encouraging ($r>1$) or
discouraging ($r<1$) lattice points on each edge as compared to the
reference case $r=1$.

Assume that $0<c_1\le n_2/n_1\le c_2<\infty$, and consider the
standard scaling transformation $S_n(x)=(x_1/n_1,x_2/n_2)$,
 $x=(x_1,x_2)\in\RR^2$. It is convenient to work with a
sup-distance between the scaled polygonal lines
$\tilde\varGamma_n:=S_n(\varGamma)$ ($\varGamma\in\CP_n$) and the
limit curve $\gamma^*$, based on the \textit{tangential
parameterization} of convex paths\vspace*{1pt} (see the \hyperref[app]{Appendix},
Section \ref{A-1}). More specifically, for $t\in[0,\infty]$ denote
by $\tilde\xi_n(t)$ the right endpoint of that part of
$\tilde\varGamma_n$ where the tangent slope (wherever it exists)
does not exceed $t$. Similarly, the tangential parameterization of
the parabola arc $\gamma^*$ [see (\ref{eq:gamma0})] is given by the
vector function
%
\begin{equation}\label{eq:g*}
g^*(t)=\biggl(\frac{t^2+2t}{(1+t)^2} ,
\frac{t^2}{(1+t)^2}\biggr),\qquad  0\le t\le\infty.
\end{equation}
The \textit{tangential distance} between $\tilde\varGamma_n$ and
$\gamma^*$ is then defined as
%
\begin{equation}\label{eq:d-T}
d_{\mathcal T}(\tilde\varGamma_n,\gamma^*):= {\sup_{0\le
t\le\infty}}|\tilde\xi_n(t)-g^*(t)|,
\end{equation}
where, as before, \mbox{$| \cdot |$} is the Euclidean vector norm in
$\RR^2$ [cf. general definition~(\ref{eq:dT}) of the metric
$d_{\mathcal T}(\cdot,\cdot)$ in the \hyperref[app]{Appendix}, Section \ref{A-1}].

We can now state our main result about the universality of the limit~sha\-pe~$\gamma^*$ under the measures $\PP_n^r$ (cf.
Theorem \ref{th:8.2a}).
\begin{theorem}\label{th:main1}
For each $r\in(0,\infty)$ and any  $\varepsilon>0$,
\[
\lim_{n\to\infty}\PP_n^r\{d_{\mathcal
T}(\tilde\varGamma_n,\gamma^*)\le\varepsilon\}=1.
\]
\end{theorem}

It can be shown (see the \hyperref[app]{Appendix}, Section \ref{A-1}) that the
Hausdorff distance $d_{\mathcal H}$ [see (\ref{eq:dH})] is dominated
by the tangential distance $d_{\mathcal T}$ defined in~(\ref{eq:dT})
(however, these metrics are not equivalent). In particular, Theorem~\ref{th:main1}
with $r=1$ recovers the limit shape result
(\ref{eq:VBS}) for the uniform distribution on $\CP_n$. As was
mentioned above, in the original paper by Sinai~\cite{S} the proof
of the limit shape result was only sketched, so even in the uniform
case our proof seems to be the first complete implementation of
Sinai's probabilistic method (which, as we will try to explain
below, is far from straightforward).

Let us also point out that Theorem \ref{th:main1} is a
\textit{nontrivial} extension of (\ref{eq:VBS}) since the measures
$\PP^r_n$ ($r\ne1$) are \textit{not close} to the uniform distribution
$\PP_n^1$ in total variation distance (denoted by \mbox{$\|\cdot\|_{\mathrm{TV}}$}),
and in fact $\|\PP_n^r-\PP_n^1\|_{\mathrm{TV}}\to1$ as $n\to\infty$ (see
Theorem \ref{th:lim(r)} in the \hyperref[app]{Appendix}).

The result of Theorem \ref{th:main1} for ``pure'' measures $\PP_n^r$
readily extends to mixed measures.
\begin{theorem}\label{th:main2}
Let  $\rho$ be a probability measure on  $(0,\infty)$,
and set
%
\begin{equation}\label{eq:rho}
\PP_n^\rho(\varGamma):=\int_0^\infty
\PP_n^r(\varGamma) \rho(\dif r), \qquad\varGamma\in\CP_n.
\end{equation}
Then, for any  $\varepsilon>0$,
\[
\lim_{n\to\infty}\PP^\rho_n \{d_{\mathcal
T}(\tilde\varGamma_n,\gamma^*)\le\varepsilon\}=1.
\]
\end{theorem}
\begin{pf}
The proof follows from equation (\ref{eq:rho}) and Theorem \ref{th:main1} by
Lebesgue's dominated convergence theorem.
\end{pf}

Theorem \ref{th:main2} shows that the limit shape result holds true
(with the same limit $\gamma^*$) when the parameter $r$ specifying
the distribution $\PP_n^{r}$\vadjust{\goodbreak} is chosen at random. Using the
terminology designed for settings with random environments, Theorems
\ref{th:main1} and \ref{th:main2} may be interpreted as ``quenched''
and ``annealed'' statements, respectively.
\begin{remark}
The universality of the limit shape $\gamma^*$, established in
Theorem~\ref{th:main1}, is not a general rule but rather an
exception, holding for some, but not all, probability measures on
the polygonal space $\CP_n$. In fact, as was\vspace*{1pt} shown by Bogachev and
Zarbaliev \cite{BZ2,BZ2a}, \textit{any} $C^3$-smooth, strictly
convex curve $\gamma$ started at the origin may appear as the limit
shape with respect to \textit{some} probability measure
$\PP_n^\gamma$ on $\CP_n$, as $n\to\infty$.
\end{remark}
\begin{remark}
The main results of the present paper have been recently reported
(without proofs) in a brief note \cite{BZ3}.
\end{remark}

\subsection{Methods}\label{sec1.3}

Our proof of Theorem \ref{th:main1} employs an elegant probabilistic
approach first applied to convex lattice polygonal lines by Sinai
\cite{S}. This method is based on randomization of the (fixed) right
endpoint $\xi=n$ of polygonal lines $\varGamma\in\CP_n$, leading to
the interpretation of a given (e.g., uniform) measure $\PP_n$ on
$\CP_n$ as the conditional distribution induced by a~suitable
probability measure $\QQ_z$ [depending on an auxiliary ``free''
parameter $z=(z_1,z_2)$] defined on the ``global'' space
$\CP=\bigcup_n\CP_n$ of \textit{all} convex lattice polygonal lines (with
finitely many edges). To make the measure $\QQ_z$ closer to $\PP_n$
on the subspace $\CP_n\subset\CP$ specified by the condition
$\xi=n$, it is natural to pick the parameter $z$ from the asymptotic
equation $\EE_z(\xi)=n(1+o(1))$ ($n\to\infty$). Then, in
principle, asymptotic properties of polygonal lines $\varGamma$
(e.g., the limit shape) can be established first for $(\CP,\QQ_z)$
and then transferred to $(\CP_n,\PP_{n})$ via conditioning with
respect to $\CP_n$ and using an appropriate local limit theorem for
the probability $\QQ_z\{\xi=n\}$. A great advantage of working with
the measure $\QQ_z$ is that it may be chosen as a ``multiplicative
statistic'' \cite{V3,V4} (i.e., a direct product of one-dimensional
probability measures), thus corresponding to the distribution of a
sequence of independent random variables, which immediately brings
in insights and well-developed analytical tools of probability
theory.

Sinai's approach in \cite{S} was motivated by a heuristic analogy
with statistical mechanics, where similar ideas are well known in
the context of asymptotic equivalence, in the thermodynamic limit,
of various statistical ensembles (i.e., \textit{microcanonical},
\textit{canonical} and \textit{grand canonical}) that may be
associated with a given physical system (e.g., gas) by optional
fixing of the total energy and/or the number of particles (see
Ruelle \cite{Ruelle}). In particular, Khinchin \cite{Kh1,Kh2} has
pioneered a systematic use of local limit theorems of probability
theory in problems of statistical mechanics. Deep connections
between statistical properties of quantum systems (where discrete
random structures naturally arise due to quantization) and
asymptotic theory of random integer partitions are discussed in
a~series of papers\vadjust{\goodbreak} by Vershik \cite{V3,V4} (see also the recent work
by Comtet et al. \cite{Comtet} and further references therein).
Note also that a general idea of randomization has proved
instrumental in a large variety of combinatorial problems (see,
e.g., \cite{AT,ABT,FG,FY,Fristedt,Kolchin,Pitman,V2} and the vast
bibliography therein).

The probabilistic method is very insightful and efficient, as it
makes the arguments heuristically transparent and natural. However,
the practical implementation of this approach requires substantial
work, especially in the two-dimensional context of convex lattice
polygonal lines as compared to the one-dimensional case exemplified
by integer partitions and the corresponding Young diagrams
\cite{V2,V3}. To begin with, evaluation of expected values and some
higher-order statistical moments of random polygonal lines leads one
to deal with various sums over the set $\calX$ of points
$x=(x_1,x_2)\in\ZZ^2_+$ with co-prime coordinates (see
Section \ref{sec2.1}). Sinai \cite{S} was able to obtain the limit
of some basic sums of such a kind by appealing to the known
asymptotic density of the set $\calX$ in $\ZZ^2_+$ (given by
$6/\pi^2$); however, this argumentation is insufficient for more
refined asymptotics. In the present paper, we handle this technical
problem by using the M\"{o}bius inversion formula (see
Section~\ref{sec3}), which enables one to reduce sums over $\calX$
to more regular sums.

As already mentioned, another crucial ingredient required for the
probabilistic method is a suitable local limit theorem that
furnishes a ``bridge'' between the global distribution $\QQ_z$ and
the conditional one, $\PP_n$. Analytical difficulties encountered in
the proof of such a result are already significant in the case of
ordinary integer partitions (for more details and concrete examples,
see \cite{ABT,FG,FVY,FY,Fristedt} and further references therein).
The case of convex lattice polygonal lines, corresponding to
two-dimensional strict vector partitions (see
Remark \ref{rm:strict}), is notoriously tedious, even though the
standard method of characteristic functions is still applicable. To
the best of our knowledge, after the original paper by Sinai
\cite{S} where the result was just stated (with a minor error in the
determinant of the covariance matrix (\cite{S}, page 111)), full
details have not been worked out in the literature (however, see
\cite{Z}). We prove the following result in this direction (cf.
Theorem~\ref{th:LCLT}).
\begin{theorem}\label{th:intro-LCLT}
Suppose that the parameter $z$ is chosen so that
$a_z:=\EE^{r}_z(\xi)=n(1+o(1))$. Then, as
$n\to\infty$,
%
\begin{equation}\label{eq:Intro-LCLT}
\QQ^{r}_z\{\xi=n\}\sim\frac{1}{2\pi (\det K_z)^{1/2}}
\exp\biggl(-{{ \frac12}}|(n-a_z)
K_z^{-1/2}|^2\biggr),
\end{equation}
where $K_z:=\Cov(\xi,\xi)$ is the covariance matrix of the random
vector $\xi$ (with respect to the probability measure
$\QQ^{r}_{z}$).
\end{theorem}
\begin{remark}
The quantities $a_z$ and $K_z$, obtained via the measure $\QQ_z^r$,
depend in general on the parameter $r$ as well. For the sake of
notational convenience the latter is omitted, which should cause no
confusion since $r$ is always fixed, unless stated explicitly
otherwise.
\end{remark}

One can\vspace*{1pt} show that the covariance matrix $K_z$ is of the order of
$|n|^{4/3}$, and in particular $\det K_z\sim\const
(n_1n_2)^{4/3}$ and
$\|K_z^{-1/2}\|=O(|n|^{-2/3})$. From the
right-hand side of (\ref{eq:Intro-LCLT}), it is then clear that one
needs to refine the error term in the asymptotic relation
$\EE^{r}_z(\xi)=n(1+o(1))$ by estimating the deviation
$\EE^{r}_z(\xi)-n$ to at least the order of $|n|^{2/3}$. We have
been able to obtain the following estimate (cf.
Theorem \ref{th:5.1}).
\begin{theorem}\label{th:intro-a}
Under the conditions of Theorem
\ref{th:intro-LCLT},
%
\begin{equation}\label{eq:intro-a}
\EE^{r}_z(\xi)=n+o(|n|^{2/3}),\qquad  n\to\infty.
\end{equation}
\end{theorem}

The proof of this result is quite involved. The main idea is to
apply the Mellin transform and use the inversion formula to obtain a
suitable integral representation for the difference $\EE^{r}_z-n$ of
the form ($j=1,2$)
%
\begin{equation}\label{eq:intro-M}\quad
\EE^{r}_z(\xi_j)-n_j=\frac{r}{2\pi i}\int_{c-i\infty}^{c+i\infty}
\frac{\widetilde{F}_j(s)\zeta(s+1)}{(-\ln
z_j)^{s+1}\zeta(s)} \,\dif s\qquad (1<c<2),
\end{equation}
where $\ln z_j$ happens to be of the order of $|n|^{-1/3}$
(according to the ``optimal'' choice of $z$ as explained at the
beginning of Section~\ref{sec1.3}; cf. Theorem~\ref{th:delta12}),$\widetilde{F}_j(s)$
is an explicit function analytic in the
strip ${1<\Re s<2}$ and~$\zeta(s)$ is the Riemann zeta function. As
usual, to obtain a better estimate of the integral one has to shift
the integration contour in (\ref{eq:intro-M}) as far to the left as
possible, and it turns out that to get an estimate of order
$o(|n|^{2/3})$ one needs to enter the critical strip
\mbox{$0<\Re s<1$}, which requires information about zeroes of the zeta
function in view of the denominator $\zeta(s)$ in
(\ref{eq:intro-M}).

\subsubsection*{Layout}

The rest of the paper is organized as follows. In
Section \ref{sec2}, we explain the basics of the probability method
in the polygonal context and define the parametric families of
measures $\QQ_z^r$ and $\PP_n^r$  ($0<r<\infty$). In
Section~\ref{sec3}, we choose suitable values of the parameter
$z=(z_1,z_2)$ (Theorem~\ref{th:delta12}), which implies convergence
of ``expected'' polygonal lines to the limit curve~$\gamma^*$
(Section \ref{sec4}, Theorems \ref{th:3.2} and \ref{th:8.1.1a}). The
refined error estimate (\ref{eq:intro-a}) is proved in
Section \ref{sec5} (Theorem \ref{th:5.1}). Higher-order moment sums
are analyzed in Section~\ref{sec6}; in particular, the
asymptotics of the covariance matrix $K_z$ is obtained in
Theorem \ref{th:K}. Section \ref{sec7} is devoted to the proof of
the local central limit theorem (Theorem \ref{th:LCLT}). Finally,
the limit shape result, with respect to both~$\QQ_z^r$ and
$\PP_n^r$, is proved in Section \ref{sec8} (Theorems \ref{th:8.2}
and~\ref{th:8.2a}). The \hyperref[app]{Appendix} includes necessary details about
the tangential parameterization and the tangential metric
$d_{\mathcal T}$ on the space of convex paths (Section~\ref{A-1}),
as well as a~discussion of the total variation distance between the
measures~$\PP_n^r$ ($r\ne1$) and the uniform distribution $\PP_n^1$
(Section~\ref{A-2}, Theorems~\ref{th:Vasser} and~\ref{th:lim(r)}).

\subsubsection*{Notation}

Let us fix some general notation frequently used in the paper. For
a row-vector $x=(x_1,x_2)\in\RR^2$, its Euclidean norm (length) is
denoted by $|x|:=(x_1^2+x_2^2)^{1/2}$, and\vadjust{\goodbreak} $\langle
x,y\rangle:=x y^{\topp
}=x_1y_1+x_2y_2$ is the
corresponding inner product of vectors $x,y\in\RR^2$. We denote
$\ZZ_+:=\{k\in\ZZ\dvtx k\ge0\}$,  $\ZZ_+^2:=\ZZ_+\times\ZZ_+$,
and similarly $\RR_+:=\{x\in\RR\dvtx x\ge0\}$,
$\RR_+^2:=\RR_+\times\RR_+$.

\section{Probability measures on spaces of convex polygonal lines}
\label{sec2}

\subsection{Encoding}\label{sec2.1}

As was observed by Sinai \cite{S}, one can encode convex lattice
polygonal lines via suitable integer-valued functions. More
specifically, consider the set $\calX$ of all pairs of co-prime
nonnegative integers,
%
\begin{equation}\label{eq:X}
\calX:=\{x=(x_1,x_2)\in\ZZ^2_+\dvtx \gcd(x_1,x_2)=1\},
\end{equation}
where $\gcd(\cdot,\cdot)$ stands for the greatest common divisor of
two integers. [In particular, the pairs $(0,1)$ and $(1,0)$ are
included in this set, while $(0,0)$ is not.] Let $\varPhi:=
(\ZZ_+)^{\calX}$ be the space of functions on $\calX$ with
nonnegative integer values, and consider the subspace of functions
with \textit{finite support},
\[
\varPhi_0 :=\{\nu\in\varPhi\dvtx \#(\supp\nu)<\infty\},
\]
where $ \supp\nu:=\{x\in\calX\dvtx \nu(x)>0\}$. It is
easy to see
that the space $\varPhi_0$ is in one-to-one correspondence with the
space $\CP=\bigcup_{n\in\ZZ_+^2}\CP_n$ of all (finite) convex lattice
polygonal lines
\[
\varPhi_0\ni\nu \quad\longleftrightarrow\quad \varGamma\in\CP.
\]
Indeed, let us interpret points $x\in\calX$ as radius-vectors
(pointing from the origin to~$x$). Now, for any $\nu\in\varPhi_0$,
a finite collection of nonzero vectors $\{x\nu(x),
x\in\supp\nu\}$, arranged in the order of increase of their
slope $x_2/x_1\in[0,\infty]$, determines consecutive edges of some
convex lattice polygonal line $\varGamma\in\CP$. Conversely, vector
edges of a lattice polygonal line $\varGamma\in\CP$ can be uniquely
represented in the form $xk$, with $x\in\calX$ and integer $k>0$;
setting $\nu(x):=k$ for such $x$ and zero otherwise, we obtain a
function $\nu\in\varPhi_0$. [The special case where $\nu(x)\equiv0$
for all $x\in\calX$ corresponds to the ``trivial'' polygonal line
$\varGamma_0$ with coinciding endpoints.]

That is to say, each $x\in\calX$ determines the \textit{direction} of
a potential edge, only utilized if $x\in\supp\nu$, in
which case
the value $\nu(x)>0$ specifies the \textit{scaling factor}, altogether
yielding a vector edge $x\nu(x)$; finally, assembling
all such
edges into a polygonal line is uniquely determined by the fixation
of the starting point (at the origin) and the convexity property.

Note that, according to the above construction, $\nu(x)$ has the
meaning of the number of lattice points on the edge $x\nu(x)$
(except its left endpoint). The right endpoint $\xi=\xi_\varGamma$
of the polygonal line $\varGamma\in\CP$ associated with a~configuration $\nu\in\varPhi_0$ is expressed by the formula
%
\begin{equation}\label{eq:xi}
\xi=\sum_{x\in\calX} x\nu(x).
\end{equation}

In what follows, we shall identify the spaces $\CP$ and $\varPhi_0$.
In particular, any probability measure on $\CP$ can be treated as
the distribution of a $\ZZ_+$-valued random field $\nu(\cdot)$ on
$\calX$ with almost surely (a.s.) finite support.

\subsection{Global measure  $\QQ_{z}$ and conditional
measure  $\PP_n$}\label{sec2.2}

Let $b_0,b_1,b_2,\ldots$ be a sequence of nonnegative numbers such
that $b_0>0$ (without loss of generality, we put $b_0=1$) and not
all $b_k$ vanish for $k\ge1$, and assume that the generating
function
%
\begin{equation}\label{eq:B}
\beta(s):=\sum_{k=0}^\infty b_k s^k
\end{equation}
is finite for $|s|<1$. Let $z=(z_1,z_2)$ be a two-dimensional
parameter, with $z_1,z_2\in(0,1)$. Throughout the paper, we shall
use the multi-index notation
\[
z^{x}:=z_1^{x_1}z_2^{x_2},\qquad  x=(x_1,x_2)\in\ZZ_+^2.
\]

We now define the ``global'' probability measure $\QQ_{z}$ on the
space $\varPhi=(\ZZ_+)^{\calX}$ as the distribution of a random
field $\nu=\{\nu(x),  x\in\calX\}$ with mutually independent
values and marginal distributions of the form
%
\begin{equation}\label{Q}
\QQ_{z}\{\nu(x)=k\}=\frac{b_k z^{kx}}{\beta(z^x)} ,\qquad
k\in\ZZ_+\  (x\in\calX).
\end{equation}

\begin{proposition}\label{pr:F0}
For each $z\in(0,1)^2$, the condition
%
\begin{equation}\label{N}
\tilde\beta(z):=\prod_{x\in\calX}\beta(z^{x})<\infty
\end{equation}
is necessary and sufficient in order that $\QQ_{z}(\varPhi_0)=1$.
\end{proposition}
\begin{pf}
According to (\ref{Q}), $\QQ_{z}\{\nu(x)>0\}=1-\beta(z^x)^{-1}$
($x\in\calX$). Since the random variables $\nu(x)$ are mutually
in\-de\-pendent for different $x\in\calX$, Borel--Cantelli's lemma
implies that $\QQ_{z}\{\nu\in\varPhi_0\}=1$ if and only if
\[
\sum_{x\in\calX}\biggl(1-\frac{1}{\beta(z^x)}\biggr)<\infty.
\]
In turn, the latter inequality is equivalent to (\ref{N}).
\end{pf}

That is to say, under condition (\ref{N}) a sample configuration of
the random field $\nu(\cdot)$ belongs ($\QQ_{z}$-a.s.) to the space
$\varPhi_0$ and therefore determines a~(random) finite polygonal
line $\varGamma\in\CP$. By the mutual independence of the values
$\nu(x)$, the corresponding $\QQ_z$-probability is given by
%
\begin{equation}\label{Q1}
\QQ_{z}(\varGamma) =\prod_{x\in\calX}\frac{b_{\nu(x)}
z^{x\nu(x)}}{\beta(z^x)} =\frac{b(\varGamma)z^{\xi}}{\tilde
\beta(z)} ,\qquad \varGamma\in\CP,
\end{equation}
where $\xi=\sum_{x\in\calX} x\nu(x)$ is the right
endpoint of
$\varGamma$ [see (\ref{eq:xi})] and
%
\begin{equation}\label{eq:b}
b(\varGamma):=\prod_{x\in\calX} b_{\nu(x)}<\infty,\qquad
\varGamma\in\CP.
\end{equation}
Note that the infinite product in (\ref{eq:b}) contains only
finitely many terms different from $1$, since for
$x\notin\supp\nu$ we have $b_{\nu(x)}=b_0=1$. Hence,
expression~(\ref{eq:b}) can be rewritten in a more intrinsic form [cf.
(\ref{eq:b-Gamma})]
%
\begin{equation}\label{eq:b1}
b(\varGamma):=\prod_{e_i\in\varGamma} b_{k_i},
\end{equation}
where the product is taken over all edges $e_i$ of
$\varGamma\in\CP_n$, and $k_i$ is the number of lattice points on the
edge $e_i$ except its left endpoint (see Section \ref{sec2.1}).

In particular, for the trivial polygonal line
$\varGamma_0\leftrightarrow\nu\equiv0$ formula (\ref{Q1}) yields
\[
\QQ_z(\varGamma_0)=\tilde\beta(z)^{-1}>0.
\]
Note, however, that $\QQ_z(\varGamma_0)<1$ since, due to our
assumptions, (\ref{eq:B}) implies $\beta(s)>\beta(0)=1$ for $s>0$
and hence, according to (\ref{N}), $\tilde\beta(z)>1$.

On the subspace $\CP_{n}\subset\CP$ of polygonal lines with the
right endpoint fixed at $n=(n_1,n_2)$, the measure $\QQ_{z}$ induces
the conditional distribution
%
\begin{equation}\label{Pn}
\PP_n(\varGamma):=\QQ_{z}(\varGamma|\CP_n)
=\frac{\QQ_{z}(\varGamma)}{\QQ_{z}(\CP_n)} ,\qquad
\varGamma\in
\CP_n,
\end{equation}
provided, of course, that $\QQ_{z}(\CP_n)>0$ [i.e., there is at
least one $\varGamma\in\CP_n$ with $b(\varGamma)>0$, cf.
(\ref{Q1})]. The parameter $z$ may be dropped from the notation for~$\PP_n$ due to the following fact.
\begin{proposition}\label{pr:noz}
The measure  $\PP_n$ in  (\ref{Pn}) does not depend on
$z$.
\end{proposition}
\begin{pf}
If $\CP_n\ni\varGamma\leftrightarrow\nu\in\varPhi_0$ then $\xi=n$
and hence formula (\ref{Q1}) is reduced to
\[
\QQ_{z}(\varGamma)= \frac{b(\varGamma)z^n}{\tilde
\beta(z)} ,\qquad \varGamma\in\CP_n.
\]
Accordingly, using (\ref{N}) and (\ref{Pn}) we get the expression
%
\begin{equation}\label{condP}
\PP_n(\varGamma)=\frac{b(\varGamma)}{\sum_{\varGamma'
\in\CP_n}
b(\varGamma')} ,\qquad \varGamma\in\CP_n,
\end{equation}
which is $z$-free.
\end{pf}

\subsection{Parametric families  $\{\QQ_{z}^{r}\}$
and  $\{\PP_{n}^{r}\}$}\label{sec2.3}

Let us consider a special parametric family of measures
$\{\QQ_{z}^{r}, 0<r<\infty\}$, determined by formula (\ref{Q}) with
the coefficients $b_k$ of the form
%
\begin{equation}\label{b_k}
b_k^{r}:= \pmatrix{r+k-1\cr k}=\frac{r(r+1)\cdots(r+k-1)}{k!},
\qquad  k\in\ZZ_+
\end{equation}
[note that $b_0^{r}={r-1\choose0}=1$, in accordance with our
convention in Section \ref{sec2.2}].

By the binomial expansion formula, the generating function
(\ref{eq:B}) of the sequence (\ref{b_k}) is given by
%
\begin{equation}\label{G}
\beta^r(s)=(1-s)^{-r}, \qquad |s|<1,
\end{equation}
and from (\ref{Q}) it follows that under the law $\QQ^{r}_{z}$ the
random variable $\nu(x)$ has the probability generating function
%
\begin{equation}\label{Phi}
\EE^{r}_{z}(s^{\nu(x)})
=\frac{\beta^{r}(sz^x)}{\beta^{r}(z^x)}
=\frac{(1-z^x)^r}{(1-sz^x)^r} ,\qquad 0\le s\le1.
\end{equation}
Consequently, formula (\ref{Q}) specializes to
%
\begin{equation}\label{Qb}\quad
\QQ^r_{z}\{\nu(x)=k\}=\pmatrix{r+k-1\cr k} z^{kx}(1-z^x)^r,\qquad
k\in\ZZ_+\  (x\in\calX).
\end{equation}
That is to say, with respect to the measure $\QQ_{z}^{r}$ the random
variable $\nu(x)$ has a negative binomial distribution with
parameters $r$ and $p=1-z^x$ (\cite{F}, Section~VI.8, page 165); in
particular, its expected value and variance are given by (see
\cite{F}, Section XI.2, page 269)
%
\begin{equation}\label{eq:nu}
\EE_{z}^{r}[\nu(x)]=\frac{rz^x}{1-z^x} ,\qquad
\Var[\nu(x)]=\frac{rz^x}{(1-z^x)^2} \qquad (x\in\calX).
\end{equation}

According to formulas (\ref{Pn}) and (\ref{condP}), the
corresponding conditional measure
$\PP_n^{r}(\cdot):=\QQ^{r}_{z}(\cdot|\CP_n)$ is expressed as
%
\begin{equation}\label{condP1}
\PP^{r}_n(\varGamma)=\frac{\QQ^{r}_{z}(\varGamma)}{\QQ
^{r}_{z}(\CP_n)}
=\frac{b^{r}(\varGamma)}{\sum_{\varGamma'\in\CP
_n}
b^{r}(\varGamma')} ,\qquad \varGamma\in\CP_n,
\end{equation}
where $b^{r}(\varGamma)$ is given by the general formula
(\ref{eq:b1}) specialized to the coefficients $b_k^r$ defined in
(\ref{b_k}).

In the special\vspace*{1pt} case $r=1$, we have $b_k^{1}={k\choose k}\equiv1$ so
that (\ref{Qb}) is reduced to the geometric distribution (with
parameter $p=1-z^x$)
\[
\QQ_{z}^{1}\{\nu(x)=k\}= z^{kx} (1-z^x), \qquad  k\in\ZZ_+\
(x\in\calX),
\]
whereas the conditional measure (\ref{condP1}) specifies the uniform
distribution on~$\CP_n$ (cf. \cite{S})
\[
\PP_n^1(\varGamma)=\frac{1}{\#(\CP_n)} ,\qquad
\varGamma\in
\CP_n.
\]

\begin{remark}\label{rm:r}
Since $b_{k+1}^{r}/b_k^{r}=(r+k)/(k+1)$, the sequence $\{b_k^{r}\}$
is strictly increasing or decreasing in $k$ according as $r>1$ or
$r<1$, respectively. That is to say, the measures $\QQ_{z}^r$ and
$\PP_n^r$ encourage ($r>1$) or discourage ($r<1$) lattice points on
edges, as compared to the reference case $r=1$.
\end{remark}

It is easy to see that condition (\ref{N}) is satisfied and, by
Proposition \ref{pr:F0},
\[
\QQ_{z}^{r}(\varPhi_0)=1,\qquad 0<r<\infty.\vadjust{\goodbreak}
\]
Indeed, using (\ref{G}) we have
\[
\tilde\beta^r(z)=\prod_{x\in\calX}(1-z^x)^{-r}=
\exp\biggl(-r\sum_{x\in\calX}\ln(1-z^{x})
\biggr)<\infty,
\]
whenever $\sum_{x\in\calX}\ln(1-z^{x})>-\infty$,
and the latter
condition is fulfilled since
\[
\sum_{x\in\calX} z^{x}\le
\sum_{x\in\ZZ^2_+}z^{x}=\sum_{x_1=0}^\infty
z_1^{x_1}\sum_{x_2=0}^\infty
z_2^{x_2}=\frac{1}{(1-z_1)(1-z_2)}<\infty.
\]

\section{Calibration of the parameter $z$}
\label{sec3}

In what follows, the asymptotic notation of the form $a_n\asymp b_n$
[where $n=(n_1,n_2)$] means that
\[
0<\liminf_{n_1,
n_2\to\infty}\frac{a_n}{b_n}\le\limsup_{n_1,
n_2\to\infty}
\frac{a_n}{b_n}<\infty.
\]
We also use the standard notation $a_n\sim b_n$ for $a_n/b_n\to1$ as
$n_1,n_2\to\infty$.

Throughout the paper, we shall work under the following convention
about the limit $n=(n_1,n_2)\to\infty$.
\begin{assumption}\label{as:c}
The notation $n\to\infty$ signifies that $n_1,n_2\to\infty$ in such
a way that $n_1\asymp n_2$. In particular, this implies that
$|n|=(n_1^2+n_2^2)^{1/2}\to\infty$ as $n\to\infty$, and $n_1\asymp
|n|$, $n_2\asymp|n|$.
\end{assumption}

The goal of this section is to use the freedom of the conditional
distribution $\PP_n^r(\cdot)=\QQ_{z}^{r}(\cdot |
\CP_n)$ from
the parameter $z$ (see Proposition \ref{pr:noz}) in order to better
adapt the measure $\QQ_{z}^{r}$ to the subspace $\CP_n\subset\CP$
determined by the condition $\xi=n$ [where $\xi=(\xi_1,\xi_2)$ is
defined in (\ref{eq:xi})]. To this end, it is natural to require
that the latter condition be satisfied (at least asymptotically) for
the \textit{expected} value of $\xi$ (cf. \cite{S,BZ1}). More
precisely, we will seek $z=(z_1,z_2)$ as a solution to the following
asymptotic equations:
%
\begin{equation}\label{calibr1}
\EE_{z}^{r}(\xi_1)\sim n_1, \qquad\EE_{z}^{r}(\xi_2)\sim n_2\qquad
(n\to\infty),
\end{equation}
where $\EE_{z}^{r}$ denotes expectation with respect to the
distribution $\QQ_{z}^{r}$.

From (\ref{eq:xi}), using the first formula in (\ref{eq:nu}), we
obtain
%
\begin{equation}\label{E_i}
\EE_{z}^{r}(\xi)=\sum_{x\in\calX}x\frac{rz^x}{1-z^x}=
r\sum_{k=1}^\infty\sum_{x\in\calX} x z^{kx}.
\end{equation}
Let us represent the parameters $z_{1}$, $z_{2}$ in the form
%
\begin{equation}\label{alpha}
z_j=\e^{-\alpha_j},\qquad \alpha_j=\delta_jn_j^{-1/3}\qquad
(j=1,2),
\end{equation}
where the quantities $\delta_1,\delta_2>0$ (possibly depending on
the ratio $n_2/n_1$) are presumed to be bounded from above and
separated from zero.\vadjust{\goodbreak} Hence, (\ref{E_i}) takes the form
%
\begin{equation}\label{E_i'}
\EE_{z}^{r}(\xi)=r\sum_{k=1}^\infty\sum_{x\in\calX}x
\e^{-k\langle\alpha,x\rangle}.
\end{equation}

\begin{theorem}\label{th:delta12}
Conditions  (\ref{calibr1}) are satisfied if
$\delta_1,\delta_2$ in  (\ref{alpha}) are chosen to be
%
\begin{equation}\label{delta12}
\delta_1=\kappa r^{1/3} (n_2/n_1)^{1/3},\qquad
\delta_2=\kappa r^{1/3} (n_1/n_2)^{1/3},
\end{equation}
where
$\kappa\,{:=}\,(\zeta(3)/\zeta
(2))^{1/3}$ and
$\zeta(s)\,{=}\,\sum_{k=1}^{\infty} k^{-s}$ is
the Riemann zeta \mbox{function}.
\end{theorem}
\begin{pf}
Let us prove the first of the asymptotic relations (\ref{calibr1}).
Set
%
\begin{equation}\label{f+}
f(x):=rx_1 \e^{-\langle\alpha,x\rangle},\qquad  x\in
\RR^2_+,
\end{equation}
and
\[
F^\sharp(h):=\sum_{x\in\calX} f(hx), \qquad h>0.
\]
Then we can rewrite (\ref{E_i'}) in projection to the first
coordinate as
%
\begin{equation}\label{E_1F}
\EE_{z}^{r}(\xi_1)=\sum_{k=1}^\infty\sum_{x\in\calX}
\frac{f(kx)}{k}= \sum_{k=1}^\infty\frac{F^\sharp(k)}{k}.
\end{equation}
Let us also consider the function
%
\begin{equation}\label{eq:F}
F(h):=\sum_{m=1}^\infty F^\sharp(hm) =\sum_{m=0}^\infty\sum_{x\in
\calX} f(hmx), \qquad h>0
\end{equation}
[adding terms with $m=0$ does not affect the sum, since $f(\cdot)$
vanishes at the origin]. Recalling the definition of the set $\calX$
[see (\ref{eq:X})], we note that $\ZZ^2_+$ can be decomposed as a
\textit{disjoint} union of multiples of $\calX$:
 $\ZZ^2_+=\bigsqcup_{m=0}^{\infty}
m\calX$. Hence, the
double sum in (\ref{eq:F}) is reduced to
%
\begin{eqnarray}\label{F}
F(h)&=& \sum_{x\in\ZZ^2_+}f(hx)=rh\sum_{x_1=1}^\infty
x_1\e^{-h\alpha_1x_1}\sum_{x_2=0}^\infty
\e^{-h\alpha_2x_2}\nonumber\\[-8pt]\\[-8pt]
&=&\frac{rh
\e^{-h\alpha_1}}{(1-\e^{-h\alpha_1})^2(1-\e^{-h\alpha_2})}.\nonumber
\end{eqnarray}

By the M\"{o}bius inversion formula (see \cite{HW}, Theorem 270,
page 237),
%
\begin{equation}\label{Mebius_f}\quad
F(h)=\sum_{m=1}^{\infty}F^\sharp(hm) \quad\Longleftrightarrow\quad
F^\sharp(h)=\sum_{m=1}^{\infty}\mu(m)F(hm),
\end{equation}
where $\mu(m)$ ($m\in\NN$) is the \textit{M\"{o}bius function} defined
as follows: $\mu(1)=1$,  $\mu(m)=(-1)^{d}$ if $m$ is a product of
$d$ different primes\vadjust{\goodbreak} and $\mu(m)=0$ if $m$ has a squared factor
(\cite{HW}, Section 16.3, page 234); in particular, $|\mu(\cdot)|\le
1$. A
sufficient condition for (\ref{Mebius_f}) is that the double series
${\sum_{k,m}}|F^\sharp(hkm)|$ should be convergent,
which is
easily verified in our case: $F^\sharp(\cdot)\ge0$ and, according to
(\ref{eq:F}) and (\ref{F}),
\[
\sum_{k,m=1}^\infty F^\sharp(kmh)=\sum
_{k=1}^\infty F(kh)=
rh\sum_{k=1}^\infty\frac{k
\e^{-hk\alpha_1}}{(1-\e^{-hk\alpha_1})^2
(1-\e^{-hk\alpha_2 })}<\infty.
\]

Using (\ref{F}) and (\ref{Mebius_f}), we can rewrite (\ref{E_1F})
as
%
\begin{equation}\label{E_1}
\EE_{z}^{r}(\xi_1)=\sum_{k=1}^\infty\frac{1}{k}\sum_{m=1}^\infty
\mu(m) F(km) =\sum_{k,m=1}^\infty\frac
{rm\mu(m)
\e^{-km\alpha_1}}
{(1-\e^{-km\alpha_1 })^2(1-\e
^{-km\alpha_2 })} .\hspace*{-35pt}
\end{equation}
Note that (\ref{alpha}) and (\ref{delta12}) imply
%
\begin{equation}\label{eq:alpha2:1}
\alpha^2_1\alpha_2=\frac{r\kappa^3}{n_1} ,\qquad
\alpha_1\alpha^2_2=\frac{r\kappa^3}{n_2} ,\qquad
\alpha_2
n_2=\alpha_1 n_1,
\end{equation}
where $\kappa$ is defined in Theorem \ref{th:delta12}. Hence, we can
rewrite (\ref{E_1}) in the form
%
\begin{equation}\label{E_1'}
n_1^{-1}\EE_{z}^{r}(\xi_1)=\frac{1}{\kappa^3}\sum_{k,
m=1}^\infty
\frac{m\mu(m)\alpha_1^2\alpha_2
\e^{-km\alpha_1}}
{(1-\e^{-km\alpha_1 })^2(1-\e
^{-km\alpha_2 })} .
\end{equation}

We now need an elementary estimate, which will also be instrumental
later on.
\begin{lemma}\label{lm:2.3}
For any  $k>0$, $\theta>0$, there exists
$C=C(k,\theta)>0$ such that, for all  $t>0$,
%
\begin{equation}\label{eq:exp_bound}
\frac{\e^{-\theta
t}}{(1-\e^{-t})^k}\le\frac{C \e^{-\theta
t/2}}{t^k}.
\end{equation}
\end{lemma}

\begin{pf}
Set $g(t):=t^k\e^{-\theta t/2}(1-\e
^{-t})^{-k}$ and note
that
\[
\lim_{t\to0+}g(t)=1,\qquad \lim_{t\to\infty}g(t)=0.
\]
By continuity, the function $g(t)$ is bounded on $(0,\infty)$, and
(\ref{eq:exp_bound}) follows.
\end{pf}

By Lemma\vspace*{1pt} \ref{lm:2.3}, the general term of the series (\ref{E_1'})
is estimated, uniformly in $k$ and $m$, by
$O(k^{-3}m^{-2})$. Hence, by Lebesgue's dominated
convergence theorem one can pass to the limit in (\ref{E_1'})
termwise
%
\begin{equation}\label{zeta32}
\lim_{n\to\infty} n_1^{-1} \EE_{z}^{r}(\xi_1)=
\frac{1}{\kappa^3}\sum_{k=1}^\infty\frac{1}{k^3}\sum_{m=1}^\infty
\frac{\mu(m)}{m^2}=\frac{1}{\kappa^3}\cdot
\frac{\zeta(3)}{\zeta(2)}=1.
\end{equation}
Here the expression for the second sum (over $m$) is obtained using
the M\"{o}bius inversion formula (\ref{Mebius_f}) with
$F^\sharp(h)=h^{-2}$,  $F(h)=\sum_{m=1}^\infty
(hm)^{-2}=h^{-2}\zeta(2)$  (cf.~\cite{HW}, Theorem 287,\vadjust{\goodbreak}
page 250).

Similarly, we can check that, as $n\to\infty$,
\[
n_2^{-1}\EE_{z}^{r}(\xi_2)=\frac{1}{\kappa^3}\sum_{k,
m=1}^\infty
\frac{m\mu(m) \alpha_1\alpha^2_2
\e^{-km\alpha_2}}
{(1-\e^{-km\alpha_1 })(1-\e^{-km\alpha_2
})^2}\to1.
\]
The theorem is proved.
\end{pf}
\begin{remark}
The term $1/\zeta(2)=6/\pi^2$ appearing in formula (\ref{zeta32})
and the like, equals the asymptotic density of co-prime pairs
$x=(x_1,x_2)\in\calX$ among all integer points on $\ZZ^2_+$ (see
\cite{HW}, Theorem 459, page 409).
\end{remark}
\begin{assumption}\label{as:z}
Throughout the rest of the paper, we assume that the parameters
$z_1,z_2$ are chosen according to formulas (\ref{alpha}),
(\ref{delta12}). In particular, the measure $\QQ_z^r$ becomes
dependent on $n=(n_1,n_2)$, as well as all $\QQ_z^r$-probabilities
and mean values.
\end{assumption}

\section{Asymptotics of ``expected'' polygonal lines}\label{sec4}
$\!\!\!$For $\varGamma\in\CP$, denote by~$\varGamma(t)$ ($t\in[0,\infty]$)
the part of $\varGamma$ where the slope does not exceed $t
n_2/n_1$. Hence, the path $\tilde\varGamma_n(t)=S_n(\varGamma(t))$
serves as a tangential parameterization of the scaled polygonal line
$\tilde\varGamma_n=S_n(\varGamma)$, where
$S_n(x_1,x_2)=(x_1/n_1,x_2/n_2)$ (see Section~\ref{sec1.2}, and also
Section \ref{A-1} below). Consider the set
%
\begin{equation}\label{eq:X_n}
\calX(t):=\{x\in\calX\dvtx x_2/x_1\le tn_2/n_1\},\qquad
t\in[0,\infty].
\end{equation}
According to the association
$\CP\ni\varGamma\leftrightarrow\nu\in\varPhi_0$ described in
Section \ref{sec2.1}, for each $t\in[0,\infty]$ the polygonal line
$\varGamma(t)$ is determined by a truncated configuration
$\{\nu(x), x\in\calX(t)\}$, hence its right endpoint
$\xi(t)=(\xi_1(t),\xi_2(t))$ is given by
%
\begin{equation}\label{xi(t)}
\xi(t)=\sum_{x\in\calX(t)}x \nu(x), \qquad  t\in
[0,\infty].
\end{equation}
In particular, $\calX(\infty)=\calX$,  $\xi(\infty)=\xi$ [cf.
(\ref{eq:xi})]. Similarly to (\ref{E_i}) and (\ref{E_i'}),
%
\begin{equation}\label{eq:E(t)}
\EE_{z}^{r}[\xi(t)]= r\sum_{k=1}^\infty\sum_{x\in\calX(t)}x
\e^{-k\langle\alpha,x\rangle}, \qquad  t\in[0,\infty].
\end{equation}

Let us also set [cf. (\ref{eq:g*})]
%
\begin{equation}\label{u_12}
g^*_1(t):=\frac{t^2+2t}{(1+t)^2} ,\qquad
g^*_2(t):=\frac{t^2}{(1+t)^2},\qquad   t\in[0,\infty].
\end{equation}
As will be verified in the \hyperref[app]{Appendix} (see Section \ref{A-1}), the
vector-function $g^*(t)=(g^*_1(t),g^*_2(t))$ gives a
tangential parameterization of the parabola $\gamma^*$ defined
in (\ref{eq:gamma0}).

The goal of this section is to establish the convergence of the
(scaled) expectation $\EE_{z}^{r}[\xi(t)]$ to the limit $g^*(t)$,
first for each $t\in[0,\infty]$ (Section \ref{sec4.1}) and then
uniformly in $t\in[0,\infty]$ (Section \ref{sec4.2}).

\subsection{Pointwise convergence}\label{sec4.1}

\begin{theorem}\label{th:3.2}
For each  $t\in[0,\infty]$,
%
\begin{equation}\label{sh}
\lim_{n\to\infty} n_j^{-1}\EE_{z}^{r}[\xi_j(t)]=g^*_j(t)\qquad
(j=1,2).\vspace*{-2pt}
\end{equation}
\end{theorem}

\begin{pf}
Theorem \ref{th:delta12} implies that (\ref{sh}) holds for
$t=\infty$. Assume that \mbox{$t<\infty$}, and let $j=1$ (the case $j=2$ is
considered in a similar manner). Setting for brevity $c_n:=n_2/n_1$
and arguing as in the proof of Theorem \ref{th:delta12} [see
(\ref{E_i'}), (\ref{E_1F}) and (\ref{E_1'})], from (\ref{eq:E(t)})
we obtain
%
\begin{eqnarray}\label{Z1}
\EE_{z}^{r}[\xi_1(t)]&=&r\sum_{k,m=1}^\infty
m\mu(m)\sum_{x_1=1}^\infty
x_1\e^{-km\alpha_1x_1}\sum_{x_2=0}^{\hat x_2} \e^{-km\alpha_2x_2}
\nonumber\\[-10pt]\\[-10pt]
&=&r\sum_{k,m=1}^\infty
\frac{m\mu(m)}{1-\e^{-km\alpha_2}} \sum
_{x_1=1}^\infty
x_1\e^{-km\alpha_1x_1}\bigl(1-\e^{-km\alpha_2(\hat{x}_2+1)}\bigr),\nonumber
\end{eqnarray}
where ${\hat{x}}_2={\hat{x}}_2(t)$ denotes the integer part of
$tc_nx_1$, so that
%
\begin{equation}\label{eq:x*}
0\le tc_nx_1-\hat{x}_2<1.
\end{equation}
Aiming to replace $\hat{x}_2+1$ by $tc_n
x_1$ in
(\ref{Z1}), we recall (\ref{eq:alpha2:1}) and rewrite the sum over
$x_1$ as\vspace*{-2pt}
%
\begin{equation}\label{eq:sum+Delta}
\sum_{x_1=1}^\infty
x_1\e^{-km\alpha_1x_1}(1-\e^{-km\alpha_1tx_1
})
+\Delta_{k,m}(t,\alpha),\vspace*{-2pt}
\end{equation}
where\vspace*{-2pt}
\[
\Delta_{k,m}(t,\alpha):= \sum_{x_1=1}^\infty
x_1\e^{-km\alpha_1x_1(1+t)}
\bigl(1-\e^{-km\alpha_2(\hat{x}_2+1-t
c_n x_1)}\bigr).\vspace*{-2pt}
\]
Using that $0<\hat{x}_2+1-tc_nx_1\le1$ [see
(\ref{eq:x*})] and applying Lemma \ref{lm:2.3}, we obtain, uniformly
in $k,m\ge1$ and $t\in[0,\infty]$,
\[
0<\frac{\Delta_{k,m}(t,\alpha)}{1-\e^{-km\alpha
_2}}\le
\sum_{x_1=1}^\infty x_1\e^{-km\alpha_1x_1}=
\frac{\e^{-km\alpha_1}}{(1-\e^{-km\alpha_1 })^2}
=O(1) \frac{\e^{-m\alpha_1/2}}{(km\alpha_1
)^{2}} .
\]
Substituting this estimate into (\ref{Z1}), we see that the error
resulting from the replacement of $\hat{x}_2+1$ by $t
c_n
x_1$ is dominated by
\[
O(\alpha_1^{-2})\sum_{k=1}^\infty\frac{1}{k^{2}}
\sum_{m=1}^\infty\frac{\e^{-m\alpha_1/2}}{m}
=O(\alpha_1^{-2})
\ln (1-\e^{-\alpha_1/2})
=O(\alpha_1^{-2}\ln\alpha_1).
\]

Returning to representation (\ref{Z1}) and computing the sum in
(\ref{eq:sum+Delta}), we find
%
\begin{eqnarray}\label{UE_1}
\EE_{z}^{r}[\xi_1(t)] &=&r\sum_{k,m=1}^\infty
\frac{m\mu(m)}{1-\e^{-km\alpha_2}}\cdot
\frac{\e^{-km\alpha_1y}}{(1-\e^{-km\alpha_1y})^2}
\bigg|_{ y=1+t}^{ y=1}\nonumber\\[-10pt]\\[-10pt]
&&{}+O(\alpha_1^{-2}\ln\alpha_1).
\nonumber\vadjust{\goodbreak}
\end{eqnarray}
Passing to the limit by Lebesgue's dominated convergence theorem,
similarly to the proof of Theorem \ref{th:delta12} [cf.
(\ref{zeta32})] we get, as $n\to\infty$,
\[
n_1^{-1}\EE_{z}^{r}[\xi_1(t)]\to\frac{1}{\kappa^3}\sum
_{k=1}^\infty
\frac{1}{k^{3}}\sum_{m=1}^\infty\frac{\mu(m)}{m^{2}}
\biggl(1-\frac{1}{(1+t)^2}\biggr) =\frac{t^2+2t}{(1+t)^2} ,
\]
which coincides with $g^*_1(t)$, as claimed.
\end{pf}

\subsection{Uniform convergence}\label{sec4.2}

There is a stronger version of Theorem \ref{th:3.2}.
\begin{theorem}\label{th:8.1.1a}
Convergence in  (\ref{sh}) is uniform in
$t\in[0,\infty]$, that is,
\[
\lim_{n\to\infty}\sup_{0\le t\le\infty} |
n_j^{-1}\EE_{z}^{r}[\xi_j(t)]-g^*_j(t)|=0\qquad (j=1,2).
\]
\end{theorem}

For the proof, we need the following general lemma.
\begin{lemma}\label{lm:8.1}
Let  $\{f_n(t)\}$ be a sequence of nondecreasing functions on a~finite interval  $[a,b]$, such that, for
each  $t\in[a,b]$,
$\lim_{n\to\infty}f_n(t)=f(t)$, where  $f(t)$ is a
continuous (nondecreasing) function on
$[a,b]$. Then the convergence  $f_n(t)\to f(t)$ as
$n\to\infty$ is uniform on  $[a,b]$.
\end{lemma}
\begin{pf}
Since $f$ is
continuous on a closed interval $[a,b]$, it is uniformly
continuous. Therefore, for any $\varepsilon>0$ there exists
$\delta>0$ such that $|f(t')-f(t)|<\varepsilon$ whenever
$|t'-t|<\delta$. Let $a=t_0<t_1<\cdots<t_N=b$ be a partition such
that $\max_{1\le i\le N}(t_{i}-t_{i-1})<\delta$. Since
$\lim_{n\to\infty}f_n(t_i)=f(t_i)$ for each $i=0,1,\ldots,N$, there
exists $n^*$ such that ${\max_{0\le i\le
N}}|f_n(t_i)-f(t_i)|<\varepsilon$ for all $n\ge n^*$. By monotonicity
of $f_n$ and $f$, this implies that for any $t\in[t_i,t_{i+1}]$ and
all $n\ge n^*$
\[
f_n(t)-f(t)\le f_n(t_{i+1})-f(t_i) \le
f_n(t_{i+1})-f(t_{i+1})+\varepsilon\le2\varepsilon.
\]
Similarly, $f_n(t)-f(t)\ge-2\varepsilon$. Therefore,
${\sup_{t\in[a,b]}}|f_n(t)-f(t)|\le2\varepsilon$, and the uniform
convergence follows.
\end{pf}
\begin{pf*}{Proof of Theorem  \ref{th:8.1.1a}}
Suppose that
$j=1$ (the case $j=2$ is handled similarly). Note that for each $n$
the function
\[
f_n(t):=n_1^{-1}\EE_{z}^{r}[\xi_1(t)]=\frac{1}{n_1}
\sum_{x\in
\calX(t)} x_1\EE_{z}^{r}[\nu(x)]
\]
is nondecreasing in $t$. Therefore, by Lemma \ref{lm:8.1} the
convergence (\ref{sh}) is uniform on any interval $[0,t^*]$.
Furthermore, since $n_1^{-1}\EE_{z}^{r}[\xi_1(\infty)]\to
g^*_1(\infty)$ and the function $g^*_1(t)$ is continuous at infinity
[see (\ref{u_12})], for the proof of the uniform convergence on a
suitable interval $[t^*,\infty]$ it suffices to show that for any
$\varepsilon>0$ one can choose $t^*$ such that, for all large enough
$n_1, n_2$ and all $t\ge t^*$,
%
\begin{equation}\label{eq:E<epsilon}
n_1^{-1}\EE_{z}^{r}|\xi_1(\infty)-\xi_1(t)|\le\varepsilon.\vadjust{\goodbreak}
\end{equation}

On account of (\ref{UE_1}) we have
%
\begin{eqnarray}\label{eq:Ez}
\EE_{z}^{r}[\xi_1(\infty)-\xi_1(t)] &=&  \sum
_{k,m=1}^\infty
\frac{rm\mu(m)}{1-\e^{-km\alpha
_2}}\cdot
\frac{\e^{-km\alpha_1(1+t)}}{(1-\e^{-km\alpha_1(1+t)}
)^2}\nonumber\\[-8pt]\\[-8pt]
&&{} +O(\alpha_1^{-2}\ln\alpha_1).
\nonumber
\end{eqnarray}
Note that by Lemma \ref{lm:2.3}, uniformly in $k,m\ge1$,
\[
\frac{\e^{-km\alpha_2}}{1-\e^{-km\alpha_2}}\cdot
\frac{\e^{-km\alpha_1(1+t)}}{(1-\e^{-km\alpha_1(1+t)})^2}=
\frac{O(1)}{\alpha_1^{2}\alpha_2 (km)^{3}(1+t)^{2}} .
\]
Returning to (\ref{eq:Ez}), we obtain, uniformly in $t\ge t^*$,
\[
\alpha_1^{2}\alpha_2\EE_{z}^{r}[\xi_1(\infty)-\xi_1(t)]
=\frac{O(1)}{(1+t)^{2}} \sum_{k=1}^\infty
\frac{1}{k^{3}}\sum_{m=1}^\infty
\frac{1}{m^{2}}=\frac{O(1)}{(1+t^*)^{2}} ,
\]
whence by (\ref{alpha}) we get (\ref{eq:E<epsilon}).
\end{pf*}

\section{Further refinement}\label{sec5}
For future applications, we need to refine the asymptotic formulas
(\ref{calibr1}) by estimating the error term. The following theorem
is one of the main technical ingredients of our work.
\begin{theorem}\label{th:5.1}
Suppose that the parameter $z$ is chosen according to formulas
(\ref{alpha}), (\ref{delta12}),
so that  $\EE^{r}_z(\xi)=n(1+o(1))$ (see
Theorem \ref{th:delta12}). Then
$\EE_{z}^{r}(\xi)=n+o(|n|^{2/3})$ as $n\to\infty$.
\end{theorem}

For the proof of this theorem, some preparations are needed.

\subsection{Approximation of sums by integrals}\label{sec5.1}

Let a function $f\colon\RR^2_+\to\RR$ be continuous and absolutely
integrable on $\RR^2_+$, together with its partial derivatives
up to the second order. Set
%
\begin{equation}\label{F0}
F(h):=\sum_{x\in\ZZ^2_+}f(hx), \qquad  h>0
\end{equation}
[as verified below, the series in (\ref{F0}) is absolutely
convergent for all $h>0$], and assume that for some $\beta>2$
%
\begin{equation}\label{beta}
F(h)=O(h^{-\beta}),\qquad  h\to\infty.
\end{equation}
Consider the Mellin transform of $F(h)$ (see, e.g., \cite{Widder},
Chapter VI, Section~9),
%
\begin{equation}\label{Mel}
\widehat{F}(s):=\int_0^\infty h^{s-1}F(h) \,\dif h.
\end{equation}

\begin{lemma}\label{lm:Delta1}
Under the\vspace*{1pt} above conditions, the function  $\widehat{F}(s)$
is meromorphic in the strip  $1<\Re s<\beta$, with a\vadjust{\goodbreak}
single  (simple) pole at  $s=2$.
Moreover,~$\widehat{F}(s)$ satisfies the identity
%
\begin{equation}\label{Muntz2}
\widehat{F}(s)=\int_0^\infty h^{s-1}\Delta_f(h) \,\dif h,\qquad1<\Re s<2,
\end{equation}
where
%
\begin{equation}\label{Delta}
\Delta_f(h):=F(h)-\frac{1}{h^2}\int_{\RR^2_+} f(x) \,\dif x,\qquad h>0.
\end{equation}
\end{lemma}

\begin{remark}
Identity (\ref{Muntz2}) is a two-dimensional analogue of M\"{u}ntz's
formula for univariate functions (see \cite{Titch1},
Section 2.11, pages 28 and 29).
\end{remark}
\begin{pf*}{Proof of Lemma  \ref{lm:Delta1}}
Let a function $\phi\dvtx\RR_+\to\RR$ be continuous and continuously
differentiable, and suppose that both $\phi$ and $\phi^{\prime}$
are absolutely integrable on $\RR_+$. It follows that
$\lim_{x\to\infty} \phi(x)=0$; indeed, note that
\[
\int_0^\infty
\phi^{\prime}(x) \,\dif x=\lim_{x\to\infty}\int_0^x
\phi^{\prime}(y) \,\dif y=\lim_{x\to\infty}\phi
(x)-\phi(0),
\]
hence $\lim_{x\to\infty} \phi(x)$ exists and, since $\phi$ is
integrable, the limit must equal zero. Then the well-known
Euler--Maclaurin summation formula states that
%
\begin{equation}\label{EM1}
\sum_{j=0}^\infty\phi(hj)= \frac{1}{h}\int_0^\infty
\phi(x) \,\dif x + \int_0^\infty\tilde B_1
\biggl(\frac{x}{h}\biggr) \phi^{\prime}(x) \,\dif x,
\end{equation}
where $\tilde B_1(x):=x-[x]-1$ (cf. \cite{BR},
Section A.4, page 254).

Applying formula (\ref{EM1}) twice to the double series (\ref{F0}),
we obtain
%
\begin{eqnarray}\qquad\label{eq:Fi}
F(h)&=&\frac{1}{h^2}\int_{\RR^2_+} f(x) \,\dif x +
\frac{1}{h}\int_{\RR^2_+}  \biggl(\tilde B_1
\biggl(\frac{x_1}{h}\biggr) \,\frac{\partial f(x)}{\partial x_1}+
\tilde B_1\biggl(\frac{x_2}{h}\biggr)\,\frac{\partial f(x)}{\partial
x_2}\biggr)\,\dif x\nonumber\\[-10pt]\\[-10pt]
&&{} + \int_{\RR^2_+} \tilde B_1\biggl(\frac{x_1}{h}\biggr)\tilde
B_1\biggl(\frac{x_2}{h}\biggr)\, \frac{\partial^2 f(x)}{\partial
x_1\,\partial x_2} \,\dif x.
\nonumber
\end{eqnarray}
Since $|\tilde B_1(\cdot)|\le1$, the above conditions on the
function $f$ imply that all integrals in (\ref{eq:Fi}) exist, hence
$F(h)$ is well defined for all $h>0$. Moreover, from (\ref{eq:Fi})
it follows that
%
\begin{equation}\label{h1}
F(h)=O(h^{-2}),\qquad \Delta_f(h)=O(h^{-1})\qquad (h\to0).
\end{equation}

Estimates (\ref{beta}) and (\ref{h1}) imply that
$\widehat{F}(s)$ as defined in (\ref{Mel}) is a regular function for
$2\!<\!\Re s\!<\!\beta$. Let us now note that for such $s$ we can rewrite~(\ref{Mel})~as
%
\begin{eqnarray}\label{M(s)}\qquad
\widehat{F}(s)&=&\int_1^\infty h^{s-1}F(h) \,\dif h+
\int_0^1 h^{s-1}F(h) \,\dif h\nonumber\\
&=&\int_1^\infty h^{s-1}F(h) \,\dif h +\int_0^1
h^{s-3} \,\dif h \int_{\RR^2_+}
f(x) \,\dif x +\int_0^1 h^{s-1} \Delta_f(h) \,\dif h\\\vadjust{\goodbreak}
&=&\int_1^\infty h^{s-1}F(h) \,\dif h
+\frac{1}{s-2}\int_{\RR^2_+} f(x) \,\dif x +\int_0^1 h^{s-1}
\Delta_f(h) \,\dif h.\nonumber
\end{eqnarray}
According to condition (\ref{beta}), the first term on the
right-hand side of (\ref{M(s)}), as a function of $s$, is regular
for $\Re s<\beta$, whereas the last term is regular for $\Re s>1$ by
(\ref{h1}). Hence, formula (\ref{M(s)}) furnishes an analytic
continuation of the function $\widehat{F}(s)$ into the strip $1<\Re
s<\beta$, where it is meromorphic and, moreover, has a single
(simple) pole at point $s=2$. Finally, observing that
\[
\frac{1}{s-2}=-\int_1^\infty h^{s-3} \,\dif h,\qquad \Re s<2,
\]
and rearranging the terms in (\ref{M(s)}) using (\ref{Delta}), we
obtain (\ref{Muntz2}).
\end{pf*}
\begin{lemma}\label{lm:Delta2}
Under the conditions of Lemma \ref{lm:Delta1},
%
\begin{equation}\label{eq:inverse}
\Delta_f(h) =\frac{1}{2\pi i}\int_{c-i\infty}^{c+i\infty} h^{-s}
\widehat{F}(s) \,\dif s,\qquad 1<c<2.
\end{equation}
\end{lemma}
\begin{pf}
From (\ref{Delta}), (\ref{eq:Fi}) we have $\Delta_f(h)=O(h^{-2})$ as
$h\to\infty$. Combined with estimate (\ref{h1}) established in
the proof of Lemma \ref{lm:Delta1}, this implies that the integral
in (\ref{Muntz2}) converges absolutely in the strip $1<\Re s<2$.
Representation (\ref{eq:inverse}) then follows from (\ref{Muntz2})
by the inversion formula for the Mellin transform (see
\cite{Widder}, Theorem 9a, pages 246 and 247).
\end{pf}

\subsection{\texorpdfstring{Proof of Theorem  \protect\ref{th:5.1}}{Proof of Theorem 5.1}}
\label{sec5.2}

Let us consider the first coordinate, $\xi_1$ (for~$\xi_2$ the proof
is similar). The proof consists of several steps.

\subsubsection*{Step 1}
According to (\ref{E_1}) we have
%
\begin{equation}\label{E_z''}
\EE_{z}^{r}(\xi_1)=\sum_{k, m=1}^\infty\frac{\mu
(m)}{k}
F(km),
\end{equation}
where [see (\ref{f+}), (\ref{F})]
\begin{eqnarray*}
F(h)&=&\sum_{x\in\ZZ^2_+} f(hx)=\frac{rh
\e^{-\alpha_1h}}{(1-\e^{-\alpha_1h})^2(1-\e^{-\alpha
_2h})} ,\qquad
h>0,\\
f(x)&=&rx_1\e^{-\langle\alpha,x\rangle}, \qquad x\in\RR^2_+.\vadjust{\goodbreak}
\end{eqnarray*}
Note that
\[
\int_{\RR^2_+}f(x) \,\dif x=r\int_0^\infty
x_1\e^{-\alpha_1x_1} \,\dif x_1\int_0^\infty
\e^{-\alpha_2x_2} \,\dif x_2= \frac{r}{\alpha_1^2
\alpha_2} .
\]
Moreover, using (\ref{eq:alpha2:1}) we have
%
\begin{equation}\label{eq:int(f)}
\sum_{k, m=1}^\infty
\frac{\mu(m)}{k}\cdot\frac{r}{(km)^2\alpha_1^2\alpha_2}=
\frac{n_1}{\kappa^3}\sum_{k=1}^\infty\frac{1}{k^3}\sum
_{m=1}^\infty
\frac{\mu(m)}{m^2}=n_1\vadjust{\goodbreak}
\end{equation}
[cf. (\ref{zeta32})]. Subtracting (\ref{eq:int(f)}) from
(\ref{E_z''}), we obtain the representation
%
\begin{equation}\label{dif1}
\EE_{z}^{r}(\xi_1)-n_1=\sum_{k, m=1}^\infty
\frac{\mu(m)}{k}
\Delta_{f}(km),
\end{equation}
where $\Delta_f(h)$ is defined in (\ref{Delta}). Clearly, the
functions $f$ and $F$ satisfy the hypotheses of Lemma
\ref{lm:Delta1} (with $\beta=\infty$). Setting $c_n:=n_2/n_1$ and
using (\ref{eq:alpha2:1}), the Mellin transform of $F(h)$ defined by
(\ref{Mel}) can be represented as
\[
\widehat{F}(s)=r\alpha_1^{-(s+1)}\widetilde F(s),
\]
where
%
\begin{equation}\label{tildeM}
\widetilde{F}(s):=\int_0^\infty\frac{y^s
\e^{-y}}{(1-\e^{-y})^2  (1-\e^{-y/c_n})} \,\dif y,\qquad \Re
s>2.
\end{equation}
As a result, applying Lemma \ref{lm:Delta2} we can rewrite
(\ref{dif1}) as
%
\begin{eqnarray}\label{dif2}
\EE_{z}^{r}(\xi_1)-n_1=\frac{r}{2\pi i} \sum_{k,
m=1}^\infty
m\mu(m) \int_{c-i\infty}^{c+i\infty}
\frac{\widetilde{F}(s)}{\alpha_1^{s+1}(km)^{s+1}} \,\dif s\nonumber\\[-8pt]\\[-8pt]
&&\eqntext{(1<
c< 2).}
\end{eqnarray}

\subsubsection*{Step 2}
It is not\vspace*{2pt} difficult to find explicitly the analytic continuation of
the function $\widetilde{F}(s)$ into the domain $1<\Re s<2$. To this
end, let us represent~(\ref{tildeM}) as
%
\begin{equation}\label{tildeM1}
\widetilde{F}(s)=J(s)+c_n\int_0^\infty\frac{y^{s-1}\e^{-y}}
{(1-\e^{-y})^2} \,\dif y+\frac{1}{2}\int_0^\infty\frac{y^s
\e^{-y}}{(1-\e^{-y})^2} \,\dif y,
\end{equation}
where
%
\begin{equation}\label{J}
J(s):=\int_0^\infty\frac{y^s\e^{-y}}
{(1-\e^{-y})^2}\biggl(\frac{1}{1-\e^{-y/c_n}}-\frac{c_n}{y}-\frac
12\biggr)
\,\dif y.
\end{equation}
The last two integrals in (\ref{tildeM1}) are easily evaluated
%
\begin{eqnarray}\label{in_1}\hspace*{28pt}
\int_0^\infty\frac{y^{s-1}\e^{-y}}{(1-\e^{-y})^2} \,\dif y
&=&\int_0^\infty y^{s-1}\sum_{k=1}^\infty k\e^{-ky}
\,\dif y
=\sum_{k=1}^\infty k\int_0^\infty y^{s-1}\e^{-ky} \,\dif
y\nonumber\\[-8pt]\\[-8pt]
&=&\sum_{k=1}^\infty\frac{1}{k^{s-1}} \int_0^\infty
u^{s-1}\e^{-u} \,\dif u =\zeta(s-1) \Gamma(s),\nonumber
\end{eqnarray}
where $\Gamma(s)=\int_0^\infty u^{s-1}\e^{-u} \,\dif u$ is the
gamma function, and similarly
%
\begin{equation}\label{in_2}
\int_0^\infty\frac{y^s\e^{-y}}{(1-\e^{-y})^2} \,\dif y=
\zeta(s) \Gamma(s+1).
\end{equation}
Substituting expressions (\ref{in_1}) and (\ref{in_2}) into
(\ref{tildeM1}), we obtain
%
\begin{equation}\label{tildeM2}
\widetilde{F}(s)=J(s)+
c_n\zeta(s-1) \Gamma(s)+\tfrac{1}{2}
\zeta(s) \Gamma(s+1).
\end{equation}

Since the expression in the parentheses in (\ref{J}) is $O(y)$ as
$y\to0$ and $O(1)$ as $y\to\infty$, the integral in (\ref{J}) is
absolutely convergent [and therefore the function $J(s)$ is regular]
for $\Re s>0$. Furthermore, it is well known that $\Gamma(s)$ is
analytic for $\Re s>0$ (\cite{Titch2}, Section 4.41, page 148), while
$\zeta(s)$ has a~single pole at point $s=1$ (\cite{Titch2},
Section 4.43, page 152). Hence, the right-hand side of (\ref{tildeM2}) is
meromorphic in the half-plane $\Re s>0$ with poles at $s=1$ and
$s=2$.

\subsubsection*{Step 3}
Let us estimate the function $\widetilde{F}(c+it)$ as $t\to\infty$.
First, by integration by parts in (\ref{J}) it is easy to show
that, uniformly in a strip $0<c_1\le\sigma\le c_2<\infty$,
%
\begin{equation}\label{J(it)}
J(\sigma+it)= O(|t|^{-2}),\qquad  t\to\infty.
\end{equation}
The gamma function in such a strip is known to satisfy a uniform
estimate
%
\begin{equation}\label{GFE}
\Gamma(\sigma+it)=O(1) |t|^{\sigma-(1/2)}
\e^{-\pi|t|/2}, \qquad t\to\infty
\end{equation}
(see \cite{Titch2}, Section 4.42, page 151). Furthermore, the zeta
function is obviously bounded in any half-plane $\sigma\ge
c_1>1$
%
\begin{equation}\label{zeta_2}
|\zeta(\sigma+it)|\le\sum_{n=1}^\infty
\frac{1}{|n^{\sigma+it}|} =\sum_{n=1}^\infty
\frac{1}{n^{\sigma}}\le
\sum_{n=1}^\infty\frac{1}{n^{c_1}}=O(1).
\end{equation}
We also have the following bounds, uniform in $\sigma$, on the
growth of the zeta function as $t\to\infty$ (see \cite{Iv},
Theorem 1.9, page 25):
%
\begin{equation}\label{zeta_123}
\zeta(\sigma+it)=\cases{
O(\ln |t|), &\quad $1\le
\sigma\le2$,\vspace*{2pt}\cr
O\bigl(t^{(1-\sigma)/2} \ln |t|\bigr), &\quad
$0\le\sigma\le1$,\vspace*{2pt}\cr
O(t^{1/2-\sigma} \ln |t|), &\quad $\sigma
\le0$.}
\end{equation}
As a result, by (\ref{GFE}), (\ref{zeta_2}) and (\ref{zeta_123}) the
second and third summands on the right-hand side of (\ref{tildeM2})
give only exponentially small contributions as compared to
(\ref{J(it)}), so that
%
\begin{equation}\label{M(it)}
\widetilde{F}(c+it)= O(|t|^{-2}),\qquad  t\to\infty\
(1<c<2).\vspace*{3pt}
\end{equation}

\subsubsection*{Step 4}
In view of (\ref{M(it)}), for $1<c<2$ there is an absolute
convergence on the right-hand side of (\ref{dif2}),
\begin{eqnarray*}
&&\sum_{k, m=1}^\infty m|\mu(m)| \int_{c-i\infty
}^{c+i\infty}
\frac{|\widetilde{F}(s)|}{|\alpha_1^{s+1}(km)^{s+1} |}
 |\dif s|\\
&&\qquad\le\frac{1}{\alpha_1^{c+1}} \sum_{k=1}^\infty\frac{1}{k^{c+1}}
\sum_{m=1}^{\infty}\frac{1}{m^c}
\int_{-\infty}^{\infty}|\widetilde{F}(c+it)| \,\dif t<\infty.\vadjust{\goodbreak}
\end{eqnarray*}\vfill\eject
Hence, the summation and integration in (\ref{dif2}) can be
interchanged to yield
%
\begin{eqnarray}\label{dif3}
\EE_{z}^{r}(\xi_1)-n_1&=& \frac{r}{2\pi i}
\int_{c-i\infty}^{c+i\infty} \frac{\widetilde{F}(s)}{\alpha_1^{s+1}}
\sum_{k=1}^\infty\frac{1}{k^{s+1}}
\sum_{m=1}^\infty\frac{\mu(m)}{m^{s}} \,\dif s\nonumber\\[-8pt]\\[-8pt]
&=&\frac{r}{2\pi i}\int_{c-i\infty}^{c+i\infty}
\frac{\widetilde{F}(s)\zeta(s+1)}{\alpha_1^{s+1}\zeta
(s)} \,\dif s.\nonumber
\end{eqnarray}
While evaluating the sum over $m$ here, we used the M\"{o}bius
inversion formula~(\ref{Mebius_f}) with $F^\sharp(h)=h^{-s}$,
 $F(h)=\sum_m (hm)^{-s}=h^{-s}\zeta(s)$ (cf.~(\ref{zeta32}); see
also~\cite{HW}, Theorem 287, page 250). Substituting (\ref{tildeM2})
into (\ref{dif3}), we finally obtain
%
\begin{equation}\label{dif4}
\EE_{z}^{r}(\xi_1)-n_1=\frac{r}{2\pi i}\int_{c-i\infty}^{c+i\infty}
\varPsi(s) \,\dif s \qquad(1<c<2),
\end{equation}
where
%
\begin{equation}\label{Psi}
\varPsi(s):=\frac{\zeta(s+1)}{\alpha_1^{s+1}}\biggl[\frac{J(s)+
c_n\zeta(s-1)\Gamma(s)}{\zeta(s)}+\frac
{1}{2} \Gamma(s+1)\biggr],
\end{equation}
and the function $J(s)$ is given by (\ref{J}).

\subsubsection*{Step 5}
By the La Vall\'{e}e Poussin theorem (see \cite{Ka},
Section 4.2, Theorem~5, page 69), there exists a constant $A>0$ such that
$\zeta(\sigma+it)\ne0$ in the domain
%
\begin{equation}\label{eta}
\sigma\ge1-\frac{A}{\ln (|t|+2)}=:\eta(t),\qquad
t\in\RR.
\end{equation}
Moreover, it is known (see \cite{Titch1}, equation (3.11.8), page 60)
that in the domain (\ref{eta}) the following uniform estimate holds:
%
\begin{equation}\label{zeta_1}
\frac{1}{\zeta(\sigma+it)}=O(\ln |t|), \qquad  t\to
\infty.
\end{equation}
Without loss of generality, one can assume $A<\ln2$, so that [see
(\ref{eta})]
\[
\eta(t)\ge\eta(0)=1-\frac{A}{\ln2}>0, \qquad  t\in\RR.
\]
Therefore, $\varPsi(s)$ [see (\ref{Psi})] is regular for all
$s=\sigma+it$ such that $2>\sigma\ge\eta(t)$ ($t\in\RR$).

Let us show that the integration contour $\Re s=c$ in (\ref{dif4})
can be replaced by the curve $\sigma=\eta(t)$ ($t\in\RR$). By the
Cauchy theorem, it suffices to check that
\[
\lim_{T\to\pm\infty}\int_{\eta(T)+iT}^{c+iT}\varPsi(s) \,\dif s=0.
\]
We have
%
\begin{equation}\label{iT}
\biggl|\int_{\eta(T)+iT}^{c+iT}\varPsi(s) \,\dif s\biggr|\le
\int_{\eta(T)}^{c}|\varPsi(\sigma+iT)| \,\dif \sigma\le
\int_{\eta(0)}^{c}|\varPsi(\sigma+iT)| \,\dif \sigma.\hspace*{-25pt}
\end{equation}
In view of the remark after formula (\ref{zeta_1}), we have
$\eta(0)>0$, hence application of estimate (\ref{zeta_2}) gives,
for $s=\sigma+iT$,  $\eta(T)\le\sigma\le c$,
\[
\biggl|\frac{\zeta(s+1)}{\alpha_1^{s+1}}\biggr|\le
\frac{\zeta(\sigma+1)}{\alpha_1^{\sigma+1}}\le
\frac{\zeta(\eta(0)+1)}{\alpha_1^{c+1}}
\]
(since $\alpha_1\to0$, we may assume that $\alpha_1<1$).

To estimate the expression in the square brackets in (\ref{Psi}), we
use estimates (\ref{J(it)}), (\ref{GFE}), (\ref{zeta_123}) and
(\ref{zeta_1}). As a result, we obtain
%
\begin{equation}\label{Psi(iT)}
\varPsi(\sigma+iT)=O(|T|^{-2}\ln |T|
),\qquad
T\to\pm\infty,
\end{equation}
which implies that the right-hand side of (\ref{iT}) tends to zero
as $T\to\pm\infty$, as required. Therefore, the integral in
(\ref{dif4}) can be rewritten in the form
%
\begin{equation}\label{Lambda}
D_n:=\int_{-\infty}^\infty\varPsi\bigl(\eta(t)+it\bigr) \,\dif \bigl(\eta(t)+it\bigr).
\end{equation}

\subsubsection*{Step 6}
It remains to estimate the integral in (\ref{Lambda}) as $n\to\infty$. Let us~set
%
\begin{eqnarray}\label{Psi0}
\varPsi_0(s):\!&=& \alpha_1^{s+1}\varPsi(s)\nonumber\\[-8pt]\\[-8pt]
& = & \zeta(s+1)\biggl[\frac{J(s)+
c_n\zeta(s-1)\Gamma(s)}{\zeta(s)}+\frac
{1}{2} \Gamma(s+1)\biggr]\nonumber
\end{eqnarray}
[see (\ref{Psi})], then equation (\ref{Lambda}) is rewritten as
\[
D_n=\alpha_1^{-2}\int_{-\infty}^\infty\alpha_1^{1-\eta(t)-it}
\varPsi_0\bigl(\eta(t)+it\bigr) \bigl(\eta^{\prime
}(t)+i\bigr) \,\dif t.
\]
Using that $\alpha_1=\delta_1/n^{1/3}$ [see (\ref{alpha})], we get
%
\begin{eqnarray}\label{|Psi|}
|D_n|&=& O(n_1^{2/3}) \int_{-\infty}^\infty
\alpha_1^{1-\eta(t)}
\bigl|\varPsi_0\bigl(\eta(t)+it\bigr)\bigr| \bigl(|\eta^{\prime
}(t)|+1\bigr) \,\dif t\nonumber\\[-8pt]\\[-8pt]
&=& O(n_1^{2/3}) \int_{-\infty}^\infty
\alpha_1^{1-\eta(t)}  \bigl|\varPsi_0\bigl(\eta(t)+it\bigr)\bigr| \,\dif t,\nonumber
\end{eqnarray}
since by (\ref{eta})
\[
|\eta^{\prime }(t)|=\frac{A}{(|t|+2)\ln
^2(|t|+2)}\le
\frac{A}{2\ln^22}=O(1).
\]

Let us now note that, as $n\to\infty$, the integrand function in
(\ref{|Psi|}) tends to zero for each $t$, because $\alpha_1\to0$ and
$1-\eta(t)>0$ [see (\ref{eta})]. Finally, eligibility of passing to
the limit under the integral sign follows from Lebesgue's dominated
convergence theorem. Indeed, the integrand function in (\ref{|Psi|})
is bounded by $|\varPsi_0(\eta(t)+it)|$, and integrability of the
latter is easily checked by applying estimates (\ref{J(it)}),
(\ref{GFE}), (\ref{zeta_123}) and (\ref{zeta_1}) to expression\vadjust{\goodbreak}
(\ref{Psi0}), which yields [cf.~(\ref{Psi(iT)})]
\[
\bigl|\varPsi_0\bigl(\eta(t)+it\bigr)\bigr|=O(|t|^{-2}\ln |t|
),\qquad  t\to\pm\infty.
\]
Thus, we have shown that the integral in (\ref{|Psi|}) is $o(1)$ as
$n\to\infty$, hence $D_n=o(|n|^{2/3})$. Substituting this
estimate into (\ref{dif4}), we obtain the statement of Theorem
\ref{th:5.1}. The proof is complete.

\vspace*{-3pt}\section{Asymptotics of higher-order moments}\vspace*{-3pt}\label{sec6}

\subsection{The variance}\label{sec6.1}

According to the second formula in (\ref{eq:nu}), we have
%
\begin{equation}\label{D_z}
\Var[\nu(x)]=\frac{rz^x}{(1-z^x)^2}=r\sum_{k=1}^\infty kz^{kx}.
\end{equation}
Let $K_z:=\Cov(\xi,\xi)$ be the covariance matrix (with respect to
the measure~$\QQ_z^r$) of the random vector $\xi=\sum_{x\in\calX}
x\nu(x)$. Recalling that the random variables $\nu(x)$ are
independent for different $x\in\calX$ and using (\ref{D_z}), we see
that the elements $K_z(i,j)=\Cov(\xi_i,\xi_j)$ of the matrix $K_z$
are given by
%
\begin{equation}\label{D_zxi}
K_z(i,j)=\sum_{x\in\calX}x_i x_j \Var[\nu(x)]= r\sum_{k=1}^\infty
\sum_{x\in\calX} k x_i x_j z^{kx},\qquad i,j\in\{1,2\}.\hspace*{-20pt}
\end{equation}

\begin{theorem}\label{th:K}
As  $n\to\infty$,
%
\begin{equation}\label{eq:Sigma}
K_z(i,j)\sim\frac{(n_1n_2)^{2/3}}{r^{1/3}\kappa}  B_{ij},\qquad
i,j\in\{1,2\},
\end{equation}
where $\kappa$ is defined in Theorem \ref{th:delta12} and
the matrix $B:=(B_{ij})$ is given by
%
\begin{equation}\label{eq:Sigma1}
B=\pmatrix{2n_1/n_2&1\cr
1&2n_2/n_1}.
\end{equation}
\end{theorem}

\begin{pf}
Let us consider $K_z(1,1)$ (the other elements of $K_z$ are analyzed
in a similar manner). Substituting (\ref{alpha}) into (\ref{D_zxi}),
we obtain
%
\begin{equation}\label{D_zxi+}
K_z(1,1)=r\sum_{k=1}^\infty\sum_{x\in
\calX}kx_1^2\e^{-k\langle\alpha,x\rangle}.
\end{equation}
Using the M\"{o}bius inversion formula (\ref{Mebius_f}), similarly to
(\ref{E_1'}) expression~(\ref{D_zxi+}) can be rewritten in the form
%
\begin{eqnarray}\label{D1_1(t)}
K_z(1,1)&=&r\sum_{k, m=1}^\infty km^2\mu(m)
\sum_{x\in\ZZ^2_+}
x_1^2\e^{-km\langle\alpha,x\rangle}\nonumber\\
&=&r\sum_{k, m=1}^\infty km^2\mu(m) \sum
_{x_1=1}^\infty
x_1^2\e^{-km\alpha_1 x_1} \sum_{x_2=0}^{\infty}
\e^{-km\alpha_2 x_2}\\
&=&r\sum_{k, m=1}^\infty
\frac{km^2\mu(m)}{1-\e^{-km\alpha_2}} \sum_{x_1=1}^\infty
x_1^2\e^{-km\alpha_1x_1}.\nonumber\vadjust{\goodbreak}
\end{eqnarray}
Note also that
%
\begin{equation}\label{eq:neweq}
\sum_{x_1=1}^\infty x_1^2\e^{-km\alpha_1x_1}
=\frac{\e^{-km\alpha_1}(1+\e^{-km\alpha_1 })}{(1-\e
^{-km\alpha_1 })^3}
=\frac{O(1)}{k^3m^3\alpha_1^3} .
\end{equation}
Returning to representation (\ref{D1_1(t)}) and using
(\ref{eq:neweq}), we obtain
%
\begin{equation}\label{eq:neweq1}
\alpha_1^3 \alpha_2K_z(1,1)= r\sum
_{k, m=1}^\infty
km^2\mu(m)\frac{\alpha_1^3\alpha_2
\e^{-km\alpha_1}(1+\e^{-km\alpha_1 })}
{(1-\e^{-km\alpha_1 })^3(1-\e
^{-km\alpha_2 })} .
\end{equation}
By Lemma \ref{lm:2.3}, the general term in the series
(\ref{eq:neweq1}) admits a uniform estimate $O(k^{-3}
m^{-2})$. Hence, by Lebesgue's dominated convergence theorem
one can pass to the limit in (\ref{eq:neweq1}) to obtain
\[
\alpha_1^3 \alpha_2K_z(1,1) \to
\frac{2r\zeta(3)}{\zeta(2)}=2r\kappa^3,\qquad
\alpha_1,\alpha_2\to0.
\]
Using (\ref{alpha}) and (\ref{delta12}), this yields
\[
K_z(1,1)\sim\frac{2 n_1/n_2}{r^{1/3}\kappa}
 (n_1n_2)^{2/3}, \qquad n\to\infty,
\]
as required [cf. (\ref{eq:Sigma}), (\ref{eq:Sigma1})].
\end{pf}

\subsection{\texorpdfstring{Statistical moments of $\nu(x)$}{Statistical moments of nu(x)}}\label{sec6.2}

Denote
%
\begin{equation}\label{eq:nu0}
\nu_0(x):=\nu(x)-\EE_{z}^{r}[\nu(x)], \qquad x\in\calX,
\end{equation}
and for $k\in\NN$ set
%
\begin{equation}\label{eq:m-mu}
m_k(x):=\EE_{z}^{r} [\nu(x)^k],\qquad
\mu_k(x):=\EE_{z}^{r} |\nu_0(x)^k|
\end{equation}
(for notational simplicity, we suppress the dependence on $r$ and
$z$).
\begin{lemma}\label{lm:mu<m}
For each  $k\in\NN$ and all  $x\in\calX$,
%
\begin{equation}\label{eq:mu<m}
\mu_k(x)\le2^k m_k(x).
\end{equation}
\end{lemma}
\begin{pf}
Omitting for brevity the argument $x$, by Newton's binomial formula
and Lyapunov's inequality we obtain
\begin{eqnarray*}
\mu_k\le\EE_{z}^{r}[(\nu+m_1)^k]&=&\sum_{i=0}^k
\pmatrix{k\cr i} m_i m_1^{k-i}\\
&\le&\sum_{i=0}^k \pmatrix{k\cr i}m_k^{i/k}
m_k^{(k-i)/k}=2^k
m_k,
\end{eqnarray*}
and (\ref{eq:mu<m}) follows.
\end{pf}
\begin{lemma}\label{lm:6.2}
For each  $k\in\NN$, there exist positive constants
$c_k=c_{k}(r)$ and $C_k=C_{k}(r)$ such that,
for all
$x\in\calX$,
%
\begin{equation}\label{t=0}
\frac{c_{k}z^{kx}}{(1-z^x)^k}\le
m_k(x)\le\frac{C_{k}z^{x}}{(1-z^x)^k} .
\end{equation}
\end{lemma}

\begin{pf}
Fix $x\in\calX$ and let $\varphi(s)\equiv
\varphi_{\nu(x)}(s):=\EE_{z}^{r}[\e^{is\nu(x)}]$ be the
characteristic function of the random variable $\nu(x)$ with respect
to the measure~$\QQ_{z}^{r}$. From (\ref{Phi}) it follows that
%
\begin{equation}\label{p.f}
\varphi(s)=\frac{\beta^{r}(z^x\e^{is})}{\beta^{r}(z^x)}=
\frac{(1-z^x)^r}{(1-z^x\e^{is})^r} .
\end{equation}
Let us first prove that for any $k\in\NN$
%
\begin{equation}\label{f'}
(1-z^x)^{-r} \,\frac{\dif^k\varphi(s)}{\dif s^k}=i^k\sum_{j=1}^k
c_{j, k}  \frac{(z^{x}\e^{is})^j}{(1-z^x\e
^{is})^{r+j}},
\end{equation}
where  $c_{j, k}\equiv c_{j, k}(r)>0$.
Indeed, if $k=1$
then differentiation of (\ref{p.f}) yields
\[
(1-z^x)^{-r} \,\frac{\dif\varphi(s)}{\dif s}
=\frac{irz^x\e^{is}}{(1-z^x\e^{is})^{r+1}} ,
\]
which is in accordance with (\ref{f'}) if we put $c_{1,1}:=r$.
Assume now that~(\ref{f'}) is valid for some $k$. Differentiating
(\ref{f'}) once more, we obtain
\begin{eqnarray*}
(1-z^x)^{-r} \,\frac{\dif^{k+1}\varphi(s)}{\dif s^{k+1}}&=&i^{k+1}
\sum_{j=1}^{k}
c_{j, k} \frac{j(z^{x}\e^{is})^j}{(1-z^x\e
^{is})^{r+j}}\\
&&{} +i^{k+1}\sum_{j=1}^{k} c_{j, k}
\frac{(r+j) (z^{x}\e^{is})^{j+1}}{(1-z^x\e^{is})^{r+j+1}}\\
&=&i^{k+1}\sum_{j=1}^{k+1} c_{j, k+1}
\frac{(z^{x}\e^{is})^j}{(1-z^x\e^{is})^{r+j}} ,
\end{eqnarray*}
where we have set
\[
c_{j, k+1}:=\cases{
c_{1, k},&\quad $j=1$,\cr
jc_{j, k}+(r+j-1)
c_{j-1, k}, &\quad$2\le j\le k$,\cr
(r+k)c_{k, k},&\quad $j=k+1$.}
\]
Hence, by induction, formula (\ref{f'}) is valid for all $k$.

Now, by (\ref{f'}) we have
\[
m_k(x)=i^{-k}  \,\frac{\dif^k\varphi(s)}{\dif s^k} \bigg|_{s=0}
=\sum_{j=1}^k c_{j, k}  \frac{z^{jx}}{(1-z^x)^j}\le
\frac{z^x}{(1-z^x)^k}\sum_{j=1}^k c_{j, k},
\]
since $0\!<\!z^x\!<\!1$. Hence, inequalities (\ref{t=0}) hold with
$c_k\!=\!c_{k, k}$,  $C_k\!=\!\sum_{j=1}^k c_{j,
k}$.\vadjust{\goodbreak}~%
\end{pf}

\subsection{Asymptotics of the moment sums}\label{sec6.3}

\begin{lemma}\label{lm:sum(z)}
For any  $k\in\NN$ and  $\theta>0$,
%
\begin{equation}\label{eq:z^k}
\sum_{x\in\calX} |x|^k \frac{z^{\theta x}}{(1-z^x)^k}\asymp
|n|^{(k+2)/3}, \qquad n\to\infty.
\end{equation}
\end{lemma}

\begin{pf}
Using (\ref{alpha}), by Lemma \ref{lm:2.3} we have
%
\begin{equation}\label{eq:z^k1}
\frac{z^{\theta x}}{(1-z^x)^k}=
\frac{\e^{-\theta\langle\alpha,x\rangle}}{(1-\e^{-\langle\alpha
,x\rangle})^k}\le
\frac{C \e^{-(\theta/2)\langle\alpha,x\rangle
}}{\langle\alpha,x\rangle^k}\le
\frac{C \e^{-(\theta/2)\langle\alpha,x\rangle
}}{\alpha_0^k
|x|^k},
\end{equation}
where $\alpha_0:=\min\{\alpha_1,\alpha_2\}$. On the other hand,
%
\begin{equation}\label{eq:z^k2}
\frac{z^{\theta x}}{(1-z^x)^k}=
\frac{\e^{-\theta\langle\alpha,x\rangle}}{(1-\e^{-\langle\alpha
,x\rangle})^k}\ge
\frac{\e^{-\theta\langle\alpha,x\rangle}}{\langle\alpha
,x\rangle^k}\ge
\frac{\e^{-\theta\langle\alpha,x\rangle}}{|\alpha|^k
|x|^k} .
\end{equation}
Since $\alpha_0\asymp|n|^{-1/3}$ and $|\alpha|\asymp|n|^{-1/3}$,
from (\ref{eq:z^k1}) and (\ref{eq:z^k2}) we see that for the proof
of (\ref{eq:z^k}) it remains to show
%
\begin{equation}\label{eq:z^k3}
\sum_{x\in\calX} \e^{-\langle\alpha,x\rangle}\asymp|n|^{2/3},
 \qquad n\to\infty.
\end{equation}

Using the M\"{o}bius inversion formula (\ref{Mebius_f}), similarly as
in Sections \ref{sec3} and~\ref{sec4} we obtain
%
\begin{eqnarray}\label{eq:m}
\sum_{x\in\calX}
\e^{-\langle\alpha,x\rangle}&=&\sum_{m=1}^\infty\mu(m)
\sum_{x\in\ZZ^2_+\setminus\{0\}}
\e^{-m\langle\alpha,x\rangle}\nonumber\\
&=&\sum_{m=1}^\infty\mu(m)
\biggl(\frac{1}{(1-\e^{-m\alpha_1})(1-\e^{-m\alpha_2})}-1\biggr)\\
&=&\sum_{m=1}^\infty
\mu(m) \frac{\e^{-m\alpha_1}+\e^{-m\alpha_2}-
\e^{-m(\alpha_1+\alpha_2)}}{(1-\e^{-m\alpha_1})(1-\e^{-m\alpha
_2})} .\nonumber
\end{eqnarray}
By Lemma \ref{lm:2.3}, the general term of the series (\ref{eq:m})
is $O(\alpha_1^{-1}\alpha_2^{-1}m^{-2})$ (uniformly in
$m$). Hence, by Lebesgue's dominated convergence theorem, the
right-hand side of (\ref{eq:m}) is asymptotically equivalent to
\[
\frac{1}{\alpha_1\alpha_2}\sum_{m=1}^\infty\frac{\mu(m)}{m^2}
=\frac{1}{\alpha_1\alpha_2\zeta(2)}\asymp|n|^{2/3},
\]
and (\ref{eq:z^k3}) follows.
\end{pf}
\begin{lemma}\label{lm:6.3}
For any  $k\in\NN$,
\[
\sum_{x\in\calX} |x|^k m_k(x)\asymp|n|^{(k+2)/3},\qquad
n\to\infty.
\]
\end{lemma}
\begin{pf}
The proof readily follows from the estimates (\ref{t=0}) and Lem\-ma~\ref{lm:sum(z)}.
\end{pf}
\begin{lemma}\label{lm:muk}
For any integer  $k\ge2$,
\[
\sum_{x\in\calX}|x|^k\mu_k(x)\asymp|n|^{(k+2)/3},\qquad
n\to\infty.
\]
\end{lemma}

\begin{pf}
An upper bound follows (for all $k\ge1$) from inequality
(\ref{eq:mu<m}) and Lemma \ref{lm:6.3}. On the other hand, by
Lyapunov's inequality and formula~(\ref{D_z}), for any $k\ge2$ we
have
\[
\mu_k(x)\ge\mu_2(x)^{k/2}= (\Var[\nu(x)])^{k/2}=
\frac{r^{k/2}z^{kx/2}}{(1-z^x)^k} ,
\]
and a lower bound follows by Lemma \ref{lm:sum(z)}.
\end{pf}
\begin{lemma}\label{lm:6.6}
For each  $k\in\NN$ and  $j=1,2$
\[
\EE_{z}^{r} [\xi_j-\EE_{z}^{r}(\xi_j)
]^{2k}
=O(|n|^{4k/3}),\qquad  n\to\infty.
\]
\end{lemma}
\begin{pf}
Let $j=1$ (the case $j=2$ is considered similarly). Using the
notation (\ref{eq:nu0}), we obtain, by the multinomial expansion,
%
\begin{eqnarray}\label{eq:sum-prod}\quad
&&\EE_{z}^{r} [\xi_1-\EE_{z}^{r}(\xi
_1)]^{2k} \nonumber\\
&&\qquad=
\EE_{z}^{r}\biggl( \sum_{x\in\calX}
x_1\nu_0(x)\biggr)^{2k}\\
&&\qquad=\sum_{\ell=1}^{2k}
\mathop{\sum_{k_1,\ldots,k_\ell\ge1,}}_{k_1 +\cdots+k_\ell=2k}
C_{k_1 ,\ldots,k_\ell}
\sum_{\{x^{1},\ldots,
x^{\ell}\}\subset\calX}
\prod_{i=1}^\ell(x_1^i)^{k_i}
\EE_{z}^{r} [\nu_0(x^i)^{k_i}],\nonumber
\end{eqnarray}
where $C_{k_1 ,\ldots,k_\ell}$ are
combinatorial coefficients
accounting for the number of identical terms in the expansion. Using
that $\EE_{z}^{r}[\nu_0(x)]=0$, we can assume that $k_i\ge2$ for
all $i=1,\ldots,\ell$. Since $k_1+\cdots+k_\ell=2k$, this implies
that $\ell\le k$. Hence, recalling the notation (\ref{eq:m-mu}) and
using Lemma \ref{lm:muk}, we see that the internal sum in
(\ref{eq:sum-prod}) (over $\{x^1,\ldots,x^\ell\}\subset\calX$) is
bounded by
\begin{eqnarray*}
\sum_{\{x^{1},\ldots,
x^{\ell}\}\subset\calX}
\prod_{i=1}^\ell|x^i|^{k_i}
\mu_{k_i}(x^i)
&\le&\prod_{i=1}^\ell\sum_{x\in\calX}
|x|^{k_i} \mu_{k_i}(x)\\
&=&O(1)\prod_{i=1}^\ell|n|^{(k_i+2)/3}
=O(1)\cdot|n|^{2(k+\ell)/3}\\
&=&O(|n|^{4k/3}),
\end{eqnarray*}
and the lemma is proved.
\end{pf}

\section{Local limit theorem}\label{sec7}

As was explained in the Introduction (see Section~\ref{sec1.3}), the
role of a local limit theorem in our approach is to yield the
asymptotics of the probability
$\QQ_z^r\{\xi=n\}\equiv\QQ_z^r(\CP_n)$ appearing in the
representation of the measure $\PP^r_n$ as a conditional
distribution, $\PP_n^r(A)=\QQ_z^r(A |\CP_n)
=\QQ_z^r(A)/\QQ_z^r(\CP_n)$, $A\subset\CP_n$ [see (\ref{condP1})].

\subsection{Statement of the theorem}\label{sec7.1}

As before, we denote $a_z:=\EE_{z}^{r}(\xi)$, $K_z:= \Cov(\xi,\xi)=
\EE_{z}^{r} (\xi-a_z )^{ \topp}(\xi
-a_z )$, where the
random vector $\xi=(\xi_1,\xi_2)$ is defined in (\ref{eq:xi}). From
(\ref{D_zxi}), it is easy to see (e.g., using the Cauchy--Schwarz
inequality together with the characterization of the equality case)
that the matrix $K_z$ is positive definite; in particular, $\det
K_z>0$ and hence $K_z$ is invertible. Let $V_z=K_z^{-1/2}$ be the
(unique) square root of the matrix $K_z^{-1}$ (see, e.g.,
\cite{Bellman}, Chapter 6, Section 5, pages 93 and 94), that is, a
symmetric, positive definite matrix such that $V_z^2=K_z^{-1}$.

Denote by $f_{0,I}(\cdot)$ the density of a standard two-dimensional
normal distribution $\mathcal{N}(0,I)$ (with zero mean and identity
covariance matrix),
\[
f_{0,I}(x)=\frac{1}{2\pi} \e^{-|x|^2 /2}, \qquad x\in
\RR^2.
\]
Then the density of the normal distribution $\mathcal{N}(a_z,K_z)$
(with mean $a_z$ and covariance matrix $K_z$) is given by
%
\begin{equation}\label{eq:phi1}
f_{a_z ,K_z}(x)= (\det K_z)^{-1/2}f_{0,I}
\bigl((x-a_z)
V_z\bigr), \qquad x\in\RR^2.
\end{equation}

With this notation, we can now state our local limit theorem.
\begin{theorem}\label{th:LCLT}
Uniformly in  $m\in\ZZ^2_+$,
%
\begin{equation}\label{eq:LCLT}
\QQ_{z}^{r}\{\xi=m\}=f_{a_z ,K_z}(m)+O
(|n|^{-5/3}),\qquad
n\to\infty.
\end{equation}
\end{theorem}
\begin{remark}
Theorem \ref{th:LCLT} is a two-dimensional local central limit
theorem for the sum $\xi=\sum_{x\in\calX} x
\nu(x)$ with
independent terms whose distribution depends on a large parameter
$n=(n_1,n_2)$; however, the summation scheme is rather different
from the classic one, since the number of nonvanishing terms is not
fixed in advance, and, moreover, the summands actually involved in
the sum are determined by sampling.
\end{remark}

One implication of Theorem \ref{th:LCLT} will be particularly
useful.
\begin{corollary}\label{cor:Q} As
$n\to\infty$,
%
\begin{equation}\label{sim}
\QQ_{z}^{r}\{\xi=n\}\sim
\frac{r^{1/3}\kappa}{2\sqrt{3} \pi} (n_1n_2)^{-2/3},
\end{equation}
where  $\kappa=(\zeta(3)/\zeta(2))^{1/3}$.
\end{corollary}

Before proving the theorem, we have to make some (quite lengthy)
technical preparations, collected below in Sections
\ref{sec7.2}--\ref{sec7.4}.

\subsection{Lemmas about the matrix norm}\label{sec7.2}

The matrix norm induced by the Euclidean vector norm \mbox{$| \cdot |$}
is defined by $\|A\|:={\sup_{|x|=1}}|x A|$. It is well known that for
a (real) square matrix $A$ its norm is given by
%
\begin{equation}\label{eq:norm}
\|A\|=\sqrt{\lambda(A^{\topp}A)} ,
\end{equation}
where $\lambda(\cdot)$ is the spectral radius of a matrix, defined
to be the largest modulus of its eigenvalues (see, e.g.,
\cite{Lancaster}, Section 6.3, pages 210 and 211).

We need some general facts about the matrix norm $\| \cdot \|$.
Even though they are mostly well known, specific references are not
easy to find (cf., e.g., \cite{Bellman,Lancaster,Horn}). For the
reader's convenience, we give neat proofs of the lemmas below based
on the spectral characterization (\ref{eq:norm}).
\begin{lemma}[(cf. \cite{Halmos}, Section 22, Theorem 4, page 40)]\label{lm:7.1.1} If
$A$ is a real matrix, then $\|A^{\topp} A\|=\|A\|^2$.
\end{lemma}
\begin{pf}
The matrix $A^{\topp} A$ is symmetric and nonnegative
definite, hence, using (\ref{eq:norm}), we obtain $\|A^{\topp}
A\|=\lambda(A^{\topp}A)=\|A\|^2$, as claimed.
\end{pf}
\begin{lemma}[(cf. \cite{Horn}, Section 5.6, Problem 23, hints (2,5) and (5,2),
pages~313 and 314)]\label{lm:|A|} If  $A=(a_{ij})$ is a  real
$d\times d$ matrix, then
%
\begin{equation}\label{eq:m-norm}
\frac{1}{d}\sum_{i, j=1}^d a_{ij}^2\le
\|A\|^2\le\sum_{i, j=1}^d a_{ij}^2 .
\end{equation}
\end{lemma}

\begin{pf}
Note that $ \sum_{i, j=1}^d a_{ij}^2=\tr(A^{\topp
} A)$,
where $\tr(\cdot)$ denotes the trace, and furthermore
\[
\lambda(A^{\topp} A)\le\tr(A^{\topp}
A)\le d\cdot
\lambda(A^{\topp} A).
\]
Since $\lambda(A^{\topp} A)=\|A\|^2$ by (\ref
{eq:norm}), this
implies (\ref{eq:m-norm}).
\end{pf}

The following simple fact pertaining to dimension $d=2$ seems to be
less known.
\begin{lemma}
\label{lm:7.1.2} Let  $A$ be a symmetric  $2\times2$ matrix with
$\det A\ne0$. Then
%
\begin{equation}\label{eq:d=2}
\|A^{-1}\|=\frac{\|A\|}{| {\det A}|} .
\end{equation}
\end{lemma}

\begin{pf}
Let $\lambda_1$ and $\lambda_2$ ($|\lambda_2|\ge|\lambda_1|>0$) be
the eigenvalues of $A$, then $|{ \det A}|=|\lambda
_1|\cdot
|\lambda_2|$ and, according to (\ref{eq:norm}),
\[
\|A\|=\sqrt{\lambda(A^2)}=|\lambda_2|,\qquad
\|A^{-1}\|=\sqrt{\lambda((A^{-1})^2)}=|\lambda_1|^{-1},
\]
which makes equality (\ref{eq:d=2}) obvious.
\end{pf}

\subsection{Estimates for the covariance matrix}\label{sec7.3}

In this section, we collect some information about the asymptotic
behavior of the matrix $K_z=\Cov(\xi,\xi)$. The next lemma is a
direct consequence of Theorem \ref{th:K}.

\begin{lemma}\label{lm:detK} As
$n\to\infty$,
\[
\det K_z\sim\frac{3(n_1n_2)^{4/3}}{r^{2/3}\kappa^{2}}.
\]
\end{lemma}

Let us now estimate the norms of the matrices $K_z$ and
$V_z=K_z^{-1/2}$.

\begin{lemma}\label{lm:K_z}
As  $n\to\infty$, one has  $\|K_z\|\asymp|n|^{4/3}$.
\end{lemma}

\begin{pf}
Lemma \ref{lm:|A|} and Theorem \ref{th:K} imply
\[
\|K_z\|^2\asymp\sum_{i, j=1}^2 K_z(i,j)^2\asymp
(n_1n_2)^{4/3}\asymp|n|^{8/3}\qquad (n\to\infty),
\]
and the required estimate follows.
\end{pf}

\begin{lemma}\label{NV}
For the matrix  $V_z\!=\!K_z^{-1/2}$, one has  $\|V_z\|\!\asymp\!|n|^{-2/3}$ as  \mbox{$n\,{\to}\,\infty$}.
\end{lemma}
\begin{pf}
Using Lemmas \ref{lm:7.1.1} and \ref{lm:7.1.2} we have
\[
\|V_z\|^2=\|V_z^2\|=\|K_z^{-1}\|=\frac{\|K_z\|}{\det K_z} ,
\]
and an application of Lemmas \ref{lm:detK} and \ref{lm:K_z}
completes the proof.
\end{pf}

We also need to estimate the so-called \textit{Lyapunov coefficient}
%
\begin{equation}\label{L3}
L_z:=\|V_z\|^{3} \sum_{x\in\calX}|x|^3\mu_3(x),
\end{equation}
where $\mu_3(x)=\EE_{z}^{r}|\nu_0(x)^3|$ [see (\ref{eq:m-mu})].
\begin{lemma}\label{lm:7.1}
As  $n\to\infty$, one has  $L_z\asymp|n|^{-1/3}$.
\end{lemma}
\begin{pf}
The proof follows from (\ref{L3}) using Lemmas \ref{NV} and \ref
{lm:muk} (with
$k=3$).
\end{pf}

\subsection{Estimates of the characteristic functions}\label{sec7.4}

Recall from Section \ref{sec2.1} that, with respect to the measure
$\QQ_{z}^{r}$, the random variables $\{\nu(x)\}_{x\in\calX}$ are
independent and have negative binomial distribution with parameters
$r$ and $p=1-z^x$. In particular, $\nu(x)$ has the characteristic
function [see (\ref{p.f})]
%
\begin{equation}\label{f_nu}
\varphi_{\nu(x)}(s):=\EE_{z}^{r}
\bigl(\e^{is\nu(x)}\bigr)=\frac{(1-z^{x})^r}{(1-z^{x}
\e^{is})^r} , \qquad s\in\RR,\vadjust{\goodbreak}
\end{equation}
and hence the characteristic function
$\varphi_{\xi}(\lambda):=\EE_{z}^{r}
 (\e^{i\langle\lambda, \xi\rangle})$
of the vector
$\xi=\sum_{x\in\calX} x\nu(x)$ is given by
\[
\varphi_{\xi}(\lambda) =\prod_{x\in\calX} \varphi_{\nu
(x)}(\langle
\lambda,x\rangle)=\prod_{x\in\calX} \frac{(1-z^{x})^r}
{(1-z^{x} \e^{i\langle\lambda,x\rangle})^r}
,\qquad
\lambda\in\RR^2.
\]

\begin{lemma}\label{lm:7.2_f}
Let  $\varphi_{\nu_0(x)}(s)$ be the characteristic function of the
random variable $\nu_0(x)=\nu(x)-\EE_{z}^{r}[\nu(x)]$. Then
%
\begin{equation}\label{x.f_4}
|\varphi_{\nu_0(x)}(s)|\le
\exp\bigl\{-{ \tfrac12}\mu_2(x)s^2+
{ \tfrac13}\mu_3(x)|s|^3\bigr\},\qquad  s\in
\RR,
\end{equation}
where $\mu_k(x)=\EE_{z}^{r}|\nu_0(x)^k|$ [see
(\ref{eq:m-mu})].
\end{lemma}
\begin{pf}
Let a random variable $\nu_1(x)$ be independent of $\nu_0(x)$ and
have the same distribution, and set $\tilde{\nu}(x):=
\nu_0(x)-\nu_1(x)$. Note that $\EE_{z}^{r}[\tilde\nu(x)]=0$ and
$\Var[\tilde{\nu}(x)]= 2 \Var[\nu(x)]=2
\mu_2(x)$. We also
have the inequality
\[
\EE_{z}^{r}|\tilde{\nu}(x)^3| \le
4\EE_{z}^{r}|\nu_0(x)^3|=4\mu_3(x)
\]
(see \cite{BR}, Lemma 8.8, pages 66 and 67). Hence, by Taylor's formula,
the characteristic function of $\tilde\nu(x)$ can be represented in
the form
%
\begin{equation}\label{x.f_2}
\varphi_{\tilde{\nu}(x)}(s)=1-\mu_2(x)
s^2+{ \tfrac23} \theta\mu_3(x) s^3,
\end{equation}
where $|\theta|\le1$. Now, using\vspace*{1pt} the elementary inequality $|y|\le
\e^{(y^2-1)/2}$ and the fact that $\varphi_{\tilde{\nu}(x)}(s)=
|\varphi_{\nu_0(x)}(s)|^2$, we get
\[
\bigl|\varphi_{\nu_0(x)}(s)\bigr|
\le\exp\bigl\{{ \tfrac12}\bigl(\bigl|\varphi_{\nu
_0(x)}(s)\bigr|^2-1\bigr)\bigr\}
=\exp\bigl\{{ \tfrac12}\bigl(\varphi_{\tilde\nu
(x)}(s)-1\bigr)\bigr\},
\]
and the lemma follows by (\ref{x.f_2}).
\end{pf}

The characteristic function of the vector $\xi_0:=\xi-a_z
=\sum_{x\in\calX}x\nu_0(x)$ is given by
%
\begin{equation}\label{x.f_5}\quad
\varphi_{\xi_0}(\lambda):= \EE_{z}^{r}
\bigl(\e^{i\langle\lambda, \xi_0\rangle}\bigr)=
\prod_{x\in
\calX}\EE_{z}^{r} \bigl(\e^{i\langle\lambda,
x\rangle\nu_0(x)}\bigr)=
\prod_{x\in\calX}\varphi_{\nu_0(x)}(\langle\lambda,x\rangle).
\end{equation}

\begin{lemma}\label{lm:7.2_F}
If  $V_z=K_z^{-1/2}$, then for all  $\lambda\in\RR^2$
%
\begin{equation}\label{x.f_6}
|\varphi_{\xi_0}(\lambda V_z)|\le\exp\bigl\{-{
\tfrac12}|\lambda|^2+{ \tfrac{1}{3}}
L_z|\lambda|^3\bigr\}.
\end{equation}
\end{lemma}
\begin{pf}
Using (\ref{x.f_5}) and (\ref{x.f_4}), we obtain
%
\begin{equation}\label{x.f_6*}
|\varphi_{\xi_0}(\lambda V_z)| \le\exp\biggl\{-\frac12\sum_{x\in
\calX} \langle\lambda V_z,x\rangle^2\mu_2(x)+\frac13\sum_{x\in
\calX}|\langle\lambda V_z,x\rangle|^3\mu_3(x)\biggr\}.\hspace*{-35pt}
\end{equation}
The first sum in (\ref{x.f_6*}) is evaluated exactly as
%
\begin{eqnarray}\label{DL0}
\sum_{x\in\calX}\langle\lambda V_z,x\rangle^2\mu_2(x)&=& \Var
\langle
\lambda V_z,\xi\rangle=\lambda V_z \Cov(\xi,\xi
) (\lambda
V_z)^{ \topp}\nonumber\\[-8pt]\\[-8pt]
&=&\lambda V_z K_zV_z\lambda^{ \topp} =
|\lambda|^2,
\nonumber
\end{eqnarray}
since $\Cov(\xi,\xi)=K_z=V_z^{-2}$. For the second sum in
(\ref{x.f_6*}), by the Cauchy--Schwarz inequality and on account of
(\ref{L3}) we have
%
\begin{equation}\label{DL}
\sum_{x\in\calX}|\langle\lambda V_z,x\rangle|^3\mu
_3(x) \le
|\lambda|^3 \|V_z\|^3 \sum_{x\in\calX}|x|^3
\mu_3(x)=|\lambda|^3L_z.
\end{equation}
Now, substituting (\ref{DL0}), (\ref{DL}) into (\ref{x.f_6*}), we
get (\ref{x.f_6}).
\end{pf}
\begin{lemma}\label{lm:7.2}
If  $|\lambda|\le L_z^{-1}$, then
%
\begin{equation}\label{16L}
\bigl|\varphi_{\xi_0}(\lambda V_z)-\e^{-|\lambda|^2
/2}\bigr|
\le16L_z|\lambda|^3 \e^{-|\lambda|^2/6}.
\end{equation}
\end{lemma}

\begin{pf}
Let us first suppose that $\frac12 L_z^{-1/3}\le|\lambda|\le
L_z^{-1}$. Then $\frac18\le L_z|\lambda|^3\le|\lambda|^2$, so
(\ref{x.f_6}) implies $|\varphi_{\xi_0}(\lambda V_z)|\le
\e^{-|\lambda|^2/6}$. Hence,
\begin{eqnarray*}
\bigl|\varphi_{\xi_0}(\lambda V_z)-\e^{-|\lambda|^2
/2}\bigr|
&\le&|\varphi_{\xi_0}(\lambda V_z)|+\e^{-|\lambda|^2
/2}\\
&\le&2 \e^{-|\lambda|^2/6}\le16
L_z|\lambda|^3 \e^{-|\lambda|^2/6}
\end{eqnarray*}
in accord with (\ref{16L}).

Suppose now that $|\lambda|\le\frac12 L_z^{-1/3}$. Taylor's formula
implies
%
\begin{equation}\label{eq:f-1}
\varphi_{\nu_0(x)}(s)-1=-{ \tfrac12}\mu_2(x)
s^2+{ \tfrac16} \theta_x\mu_3(x) s^3,
\end{equation}
where $|\theta_x|\le1$. By Lyapunov's inequality, $\mu_2(x)\le
\mu_3(x)^{2/3}$, so
%
\begin{equation}\label{eq:1/2}
\bigl|\varphi_{\nu_0(x)}(s)-1\bigr|\le{ \tfrac{1}{2}}
|s|^2\mu_3(x)^{2/3}+ { \tfrac{1}{6}}|s|^3 \mu_3(x).
\end{equation}
For $s=\langle\lambda V_z,x\rangle$, we have from (\ref{DL})
%
\begin{equation}\label{eq:m3}
|\langle\lambda V_z,x\rangle| \mu_3(x)^{1/3} \le
L_z^{1/3}|\lambda|\le{ \tfrac{1}{2}},
\end{equation}
and so (\ref{eq:1/2}) yields
%
\begin{equation}\label{eq:|f-1|}
\bigl|\varphi_{\nu_0(x)}(\langle\lambda V_z,x\rangle)-1\bigr| \le
{ \tfrac12\cdot\tfrac14+\tfrac{1}{6}\cdot\tfrac18<\tfrac
{1}{2}}.
\end{equation}
Similarly, using the elementary inequality $(a+b)^2\le
2(a^2+b^2)$, from (\ref{eq:1/2}) we obtain
\[
\bigl|\varphi_{\nu_0(x)}(s)-1\bigr|^2\le
{ \tfrac{1}{2}} \bigl(|s|\mu
_3(x)^{1/3}+
{ \tfrac{1}{9}}|s|^3\mu_3(x)\bigr)
|s|^3\mu_3(x),
\]
whence, in view of (\ref{eq:m3}) and a general bound $|\langle
\lambda V_z,x\rangle|\le|\lambda| \cdot \|V_z\| \cdot |x|$,
it follows that
%
\begin{eqnarray}\label{|f-1|^2}
\bigl|\varphi_{\nu_0(x)}(\langle\lambda V_z,x\rangle)-1
\bigr|^2&\le&
{ \tfrac12}\bigl({ \tfrac12+\tfrac{1}{9}\cdot
\tfrac18}\bigr)
|\lambda|^3\|V_z\|^3|x|^3\mu_3(x)\nonumber\\[-8pt]\\[-8pt]
&\le&{ \tfrac13}|\lambda|^3\|
V_z\|^3|x|^3\mu_3(x).
\nonumber
\end{eqnarray}

Consider the function $\ln (1+y)$ of complex variable $y$,
choosing the principal branch of the logarithm (i.e., such that $\ln
1=0$). Taylor's expansion implies $\ln (1+y)=y+\theta
y^2$ for
$|y|\le{ \frac12}$, where $|\theta|\le1$. By
(\ref{eq:f-1}), (\ref{eq:|f-1|}) and~(\ref{|f-1|^2}) this yields
\[
\ln\varphi_{\nu_0(x)}(\langle\lambda V_z,x\rangle)=
-{ \tfrac12}\langle\lambda V_z,x\rangle^2\mu_2(x)+
{ \tfrac12}\tilde\theta_x|\lambda|^3
\|V_z\|^3|x|^3\mu_3(x),
\]
where $|\tilde\theta_x|\le1$. Substituting this into (\ref{x.f_5}),
due to (\ref{DL0}) and (\ref{DL}) we obtain
\[
\ln\varphi_{\xi_0}(\lambda V_z) = \sum_{x\in\calX} \ln
\varphi_{\nu_0}(\langle\lambda
V_z,x\rangle)=-{ \frac12}|\lambda|^2
+{ \frac12} \theta_1 L_z|\lambda|^3\qquad
(|\theta_1|\le1).
\]
Using the elementary inequality $|\e^{y}-1|\le|y|\e^{|y|}$,
which holds for any $y\in\CC$, we have
\begin{eqnarray*}
\bigl|\varphi_{\xi_0}(\lambda V_z)-\e^{-|\lambda|^2
/2}\bigr|
&=&\e^{-|\lambda|^2 /2}
\bigl|\e^{\theta_1L_z|\lambda|^3 /2}-1\bigr|\\
&\le&\e^{-|\lambda|^2 /2}\cdot{ \tfrac
12}
L_z|\lambda|^3 \e^{L_z|\lambda|^3 /2} \le\e
^{-|\lambda|^2 /2}
L_z|\lambda|^3,
\end{eqnarray*}
and the proof is complete.
\end{pf}
\begin{lemma}\label{lm:7.3}
For all  $\lambda\in\RR^2$,
%
\begin{equation}\label{f_J}
|\varphi_{\xi_0}(\lambda)|\le\exp\bigl\{-{ \tfrac
14} r
J_{\alpha}(\lambda)\bigr\},
\end{equation}
where
%
\begin{equation}\label{J_0}
J_{\alpha}(\lambda):=\sum_{x\in\calX}
\e^{-\langle\alpha, x\rangle}(1-\cos\langle\lambda,x\rangle).
\end{equation}
\end{lemma}
\begin{pf}
According to (\ref{x.f_5}), we have
%
\begin{equation}\label{eq:J_1}
|\varphi_{\xi_0}(\lambda)|=|\varphi_{\xi}(\lambda)|
=\exp\biggl\{ \sum_{x\in\calX}
\ln\bigl|\varphi_{\nu(x)}(\langle\lambda, x\rangle)
\bigr|\biggr\}.
\end{equation}
Using (\ref{f_nu}), for any $s\in\RR$ we can write
\begin{eqnarray*}
\ln\bigl|\varphi_{\nu(x)}(s)\bigr|&=&\frac{r}{2}\ln\frac
{|1-z^x|^2}{|1-z^x\e^{is}|^{2}}
\le\frac{r}{2}\biggl(\frac{|1-z^x|^2}{|1-z^x\e^{is}|^{2}}-1
\biggr)\\
&=& -\frac{rz^x(1-\cos s)}{|1-z^x\e^{is}|^2}\le-\frac{rz^x(1-\cos
s)}{4} .
\end{eqnarray*}
Utilizing this estimate under the sum in (\ref{eq:J_1}) (with
$s=\langle\lambda, x\rangle$) and recalling the notation
(\ref{alpha}), we arrive at (\ref{f_J}).
\end{pf}

\subsection{\texorpdfstring{Proof of Theorem \protect\ref{th:LCLT} and Corollary \protect\ref{cor:Q}}
{Proof of Theorem 7.1 and Corollary 7.2}}
\label{sec7.5}

Let us first deduce the corollary from the theorem.
\begin{pf*}{Proof of Corollary \ref{cor:Q}}
According to
Theorem \ref{th:5.1},
$a_z:=\EE_{z}^{r}(\xi)=n+o(|n|^{2/3})$. Together with
Lemma \ref{NV} this implies
\begin{eqnarray*}
|(n-a_z)V_z|&\le&|n-a_z| \cdot
\|V_z\|\\
&=&o(|n|^{2/3}) O(|n|^{-2/3})=o(1).
\end{eqnarray*}
Hence, by Lemma \ref{lm:detK} we get
\begin{eqnarray*}
f_{a_z ,K_z}(n)&=&\frac{1}{2\pi} (\det K_z)^{-1/2}
 \e^{-|(n-a_z )V_z|^2 /2}\\
&=& \frac{r^{1/3}\kappa}{2\sqrt{3} \pi} (n_1n_2)^{-2/3}\bigl(1+o(1)\bigr),
\end{eqnarray*}
and (\ref{sim}) follows from (\ref{eq:LCLT}).
\end{pf*}
\begin{pf*}{Proof of Theorem \ref{th:LCLT}}
By definition, the characteristic function of the random vector
$\xi_0=\xi-a_z$ is given by the Fourier series
\[
\varphi_{\xi_0}(\lambda):= \EE_{z}^{r}
\bigl(\e^{i\langle\lambda, \xi_0\rangle}\bigr)
=\sum_{m\in\ZZ_+^2}\QQ_{z}^{r}\{\xi=m\} \e^{i\langle
\lambda,  m-a_z \rangle},\qquad \lambda
\in\RR^2,
\]
and hence the Fourier coefficients are expressed as
%
\begin{equation}\label{l_1}
\QQ_{z}^{r}\{\xi=m\}=\frac{1}{4\pi^2}\int_{T^2} \e^{-i\langle
\lambda,  m-a_z \rangle}
\varphi_{\xi_0}(\lambda) \,\dif \lambda,\qquad  m\in\ZZ_+^2,
\end{equation}
where
$T^2:=\{\lambda=(\lambda_1,\lambda_2)\in\RR^2\dvtx|\lambda_1|\le\pi
,
|\lambda_2|\le\pi\}$. On the other hand, the characteristic function
corresponding to the normal probability density $f_{a_z ,K_z}(x)$
[see (\ref{eq:phi1})] is given by
\[
\varphi_{a_z ,K_z}(\lambda)=
\e^{i\langle\lambda, a_z \rangle
-|\lambda
V_z^{-1} |^2 /2},\qquad
\lambda\in\RR^2,
\]
so by the Fourier inversion formula
%
\begin{equation}\label{f_o}
f_{a_z ,K_z}(m)= \frac{1}{4\pi^2}
\int_{\RR^2}\e^{-i\langle\lambda,  m-a_z
\rangle-|\lambda
V_z^{-1} |^2 /2} \,\dif
 \lambda,\qquad
m\in\ZZ^2_+.
\end{equation}

Note that if $|\lambda V_z^{-1} |\le L_z^{-1}$ then,
according to
Lemmas \ref{NV} and \ref{lm:7.1},
\[
|\lambda|\le|\lambda V_z^{-1} |\cdot\|V_z\| \le L_z^{-1}
\|V_z\|=O(|n|^{-1/3})=o(1),
\]
which implies that $\lambda\in T^2$. Using this observation and
subtracting (\ref{f_o}) from (\ref{l_1}), we get, uniformly in
$m\in\ZZ^2_+$,
%
\begin{equation}\label{I}
|\QQ_{z}^{r}\{\xi=m\}-f_{a_z ,K_z}(m)|\le
I_1+I_2+I_3,
\end{equation}
where
\begin{eqnarray*}
I_1&:=&\frac{1}{4\pi^2}\int_{\{\lambda \dvtx |\lambda V_z^{-1}
|\le L_z^{-1}\}}
\bigl|\varphi_{\xi_0}(\lambda)-\e^{-|\lambda V_z^{-1}
|^2 /2}
\bigr| \,\dif \lambda,\\
I_2&:=&\frac{1}{4\pi^2}\int_{\{\lambda \dvtx |\lambda V_z^{-1}
|>L_z^{-1}\}}
\e^{-|\lambda V_z^{-1} |^2
/2} \,\dif \lambda,\\
I_3&:=&\frac{1}{4\pi^2}
\int_{T^2\cap\{\lambda \dvtx |\lambda V_z^{-1}
|>L_z^{-1}\}}
|\varphi_{\xi_0}(\lambda)| \,\dif \lambda.
\end{eqnarray*}

By the substitution $\lambda=yV_z$, the integral $I_1$ is
reduced to
%
\begin{eqnarray}\label{I1}
I_1&=&\frac{| {\det V_z}|}{4\pi^2} \int_{|y|\le L_z^{-1}}
\bigl|\varphi_{\xi_0}(y V_z)-
\e^{-|y|^2 /2}\bigr| \,\dif y\nonumber\\
&=&O(1) (\det K_z)^{-1/2}L_z\int_{\RR^2}
|y|^3\e^{-|y|^2 /6} \,\dif y\\
&=& O(|n|^{-5/3}),\nonumber
\end{eqnarray}
on account of Lemmas \ref{lm:detK}, \ref{lm:7.1} and \ref{lm:7.2}.
Similarly, again putting $\lambda=yV_z$ and passing to the
polar coordinates, we get, due to Lemmas \ref{lm:detK} and
\ref{lm:7.1},
%
\begin{eqnarray}\label{I2}
I_2&=& \frac{| {\det V_z}|}{2\pi} \int_{L_z^{-1}}^\infty
|y|
\e^{-|y|^2 /2} \,\dif |y|\nonumber\\
&=&O(|n|^{-4/3})\e^{-L_z^{-2} /2}\\
&=&o(|n|^{-5/3}).
\nonumber
\end{eqnarray}

Estimation of $I_3$ is the main part of the proof. Using Lemma
\ref{lm:7.3}, we obtain
%
\begin{equation}\label{I3}
I_3= O(1)\int_{T^2\cap\{|\lambda V_z^{-1} |>L_z^{-1}\}}
\e^{-J_{\alpha}(\lambda)} \,\dif \lambda,
\end{equation}
where $J_{\alpha}(\lambda)$ is given by (\ref{J_0}). The condition
$|\lambda V_z^{-1} |>L_z^{-1}$ implies that
$|\lambda|>\eta|\alpha|$ for a suitable (small enough) constant
$\eta>0$ and hence
\[
\max\{|\lambda_1|/\alpha_1,|\lambda_2|/\alpha_2\}>\eta
\]
for otherwise from (\ref{alpha}) and
Lemmas \ref{lm:K_z} and \ref{lm:7.1}
it would follow that
\[
1<L_z|\lambda V_z^{-1} |\le L_z\eta
|\alpha| \cdot
\|K_z\|^{1/2}= O(\eta)\to0 \qquad\mbox{as } \eta\downarrow0.
\]
Hence, the estimate (\ref{I3}) is reduced to
%
\begin{equation}\label{I3.1}
I_3=O(1) \biggl(
\int_{|\lambda_1|>\eta\alpha_1}+\int_{|\lambda
_2|>\eta\alpha_2}\biggr)
\e^{-J_{\alpha}(\lambda)} \,\dif \lambda.
\end{equation}

To estimate the first integral in (\ref{I3.1}), by keeping in the
sum (\ref{J_0}) only pairs of the form $x=(x_1,1)$,
 $x_1\in\ZZ_+$, we obtain
%
\begin{eqnarray}\label{J_1}
\e^{\alpha_2}J_{\alpha}(\lambda) &\ge& \sum_{x_1=0}^\infty
\e^{-\alpha_1x_1} \bigl(1-\Re
\e^{i(\lambda_1 x_1+\lambda_2)}\bigr)\nonumber\\
&=&\frac{1}{1-\e^{-\alpha_1}}-
\Re\biggl(\frac{\e^{i\lambda_2}}{1-\e^{-\alpha_1+i\lambda
_1}}\biggr)\\
&\ge&
\frac{1}{1-\e^{-\alpha_1}}-\frac{1}{|1-\e^{-\alpha_1+i\lambda
_1}|} ,\nonumber
\end{eqnarray}
because $\Re u\le|u|$ for any $u\in\CC$. Since
$\eta \alpha_1\le|\lambda_1|\le\pi$, we have
\[
|1-\e^{-\alpha_1+i\lambda_1}|
\ge|1-\e^{-\alpha_1+i\eta\alpha
_1}|\sim\alpha_1
(1+\eta^2)^{1/2}  \qquad (\alpha_1\to0).
\]
Substituting this estimate into (\ref{J_1}), we conclude that
$J_\alpha(\lambda)$ is asymptotically bounded from below by
$C(\eta) \alpha_1^{-1} \asymp|n|^{1/3}$,
uniformly in
$\eta \alpha_1\le|\lambda_1|\le\pi$. Thus, the first integral
in (\ref{I3.1}) is bounded by
\[
O(1)\exp(-\const\cdot|n|^{1/3})=o(|n|^{-5/3}).
\]

Similarly, the second integral in (\ref{I3.1}) (where
$|\lambda_2|>\eta \alpha_2$) is estimated by reducing summation
in (\ref{J_0}) to that over $x=(1,x_2)$ only. As a result, we obtain
that $I_3=o(|n|^{-5/3})$. Substituting this estimate
together with (\ref{I1}) and (\ref{I2}) into (\ref{I}), we get
(\ref{eq:LCLT}), and so the theorem is proved.
\end{pf*}

\section{Proof of the limit shape results}\label{sec8}

Recall the notation [see (\ref{eq:X_n}), (\ref{xi(t)})]
$\xi(t)=\sum_{x\in\calX(t)}x\nu(x)$, where $\calX
(t)=\{x\in
\calX\dvtx x_2/x_1\le t(n_2/n_1)\}$, $t\in[0,\infty]$. As
stated at
the beginning of Section \ref{sec4}, the tangential parameterization
of the scaled polygonal line $\tilde\varGamma_n=S_n(\varGamma)$ is
given by
%
\begin{equation}\label{eq:tilde-xi}
\tilde\xi_n(t):=S_n(\xi(t))=(n_1^{-1}\xi_1(t),
n_2^{-1}\xi_2(t)), \qquad  t\in[0,\infty],
\end{equation}
whereas the limit shape $\gamma^*$ determined by equation
(\ref{eq:gamma0}) is parameterized by the vector-function
$g^*(t)=(g^*_1(t),g^*_2(t))$ defined in (\ref{u_12}) (see more
details in the \hyperref[app]{Appendix}, Section \ref{A-1}).

The goal of this section is to use the preparatory results obtained
so far and prove the uniform convergence of random paths
$\tilde\xi_n(\cdot)$ to the limit $g^*(\cdot)$ in probability with
respect to both $\QQ_z^r$ (Section \ref{sec8.1}) and $P_n^r$
(Section \ref{sec8.2}). Let us point out that, in view of
(\ref{eq:tilde-xi}), Theorems \ref{th:8.2} and \ref{th:8.2a} below
can be easily reformulated (cf. Theorem \ref{th:main1} stated in
the \hyperref[sec1]{Introduction}) using the tangential distance $d_{\mathcal
T}(\tilde\varGamma_n,\gamma^*)= \sup_{0\le
t\le\infty}|\tilde\xi_n(t)-g^*(t)|$ [see
(\ref{eq:d-T}); cf. general definition (\ref{eq:dT}) in Section
\ref{A-1} below].

\subsection{Limit shape under  $\QQ_{z}^{r}$}\label{sec8.1}

Let us first establish the universality of the limit shape under the
measures $\QQ_z^r$, which, in conjunction with the next Theorem
\ref{th:8.2a}, illustrates the asymptotic ``equivalence'' of the
probability spaces $(\CP,\QQ_z^r)$ and\vadjust{\goodbreak} $(\CP_n,\PP_n^r)$.
\begin{theorem}\label{th:8.2}
For each  $\varepsilon>0$,
\[
\lim_{n\to\infty} \QQ_{z}^{r}\Bigl\{{\sup_{0\le
t\le\infty}}|n_j^{-1}\xi_j(t)-g^*_j(t)|
\le\varepsilon\Bigr\}=1\qquad (j=1,2).
\]
\end{theorem}
\begin{pf}
By Theorems \ref{th:3.2} and \ref{th:8.1.1a}, the expectation of the
random process $n_j^{-1}\xi_j(t)$ uniformly\vspace*{1pt} converges to $g^*_j(t)$
as $n\to\infty$. Therefore, we only need to check that for each
$\varepsilon>0$
\[
\lim_{n\to\infty}\QQ_{z}^{r}\Bigl\{{\sup_{0\le t\le\infty}
n_j^{-1}}|\xi_j(t)-\EE_{z}^{r}[\xi_j(t)]
|>\varepsilon\Bigr\}=0.
\]

Note that the random process
$\xi_{0j}(t):=\xi_j(t)-\EE_{z}^{r}[\xi_j(t)]$ has independent
increments and zero mean; hence it is a martingale with respect to
the natural filtration ${\mathcal F}_t:=\sigma\{\nu(x),  x\in
\calX(t)\}$,  $t\in[0,\infty]$. From the definition of~$\xi_j(t)$
[see (\ref{xi(t)})], it is also clear that $\xi_{0j}(t)$ is a
c\`{a}dl\`{a}g process; that is, its paths are everywhere
right-continuous and have left limits. Therefore, applying the
Kolmogorov--Doob submartingale inequality (see, e.g.,
\cite{Yeh}, Corollary 2.1, page 14) and using Theorem~\ref{th:K}, we
obtain
\[
\QQ_{z}^{r}\Bigl\{{\sup_{0\le t\le\infty}}|\xi_{0j}(t)|>
\varepsilon n_j\Bigr\} \le\frac{\Var(\xi_j)}{(\varepsilon n_j)^2}=
O(|n|^{-2/3})\to0,
\]
and the theorem is proved.
\end{pf}

\subsection{Limit shape under  $\PP_n^r$}\label{sec8.2}

We are finally ready to prove our main result about the universality
of the limit shape under the measures $\PP_n^r$ (cf.~Theo\-rem~\ref{th:main1}).
\begin{theorem}\label{th:8.2a}
For any  $\varepsilon>0$,
\[
\lim_{n\to\infty} \PP^{r}_n\Bigl\{{\sup_{0\le t\le\infty}}
|
n_j^{-1}\xi_j(t)-g^*_j(t)|
\le\varepsilon\Bigr\}=1\qquad (j=1,2).
\]
\end{theorem}
\begin{pf}
Similarly as in the proof of Theorem \ref{th:8.2}, the claim of the
theorem is reduced to the limit
\[
\lim_{n\to\infty} \PP^{r}_n\Bigl\{{\sup_{0\le t\le\infty}}
|\xi_{0j}(t)|>\varepsilon n_j\Bigr\}=0,
\]
where $\xi_{0j}(t)=\xi_j(t)-\EE_{z}^{r}[\xi_j(t)]$. Using (\ref{Pn})
we get
%
\begin{equation}\label{A1}
\PP^{r}_n\Bigl\{{\sup_{0\le t\le\infty}}|\xi
_{0j}(t)|>
\varepsilon n_j\Bigr\} \le\frac{\QQ_{z}^{r} \{{\sup_{0\le
t\le\infty}}|\xi_{0j}(t)|> \varepsilon
n_j\}}{\QQ_{z}^{r}\{\xi=n\}} .
\end{equation}
Applying the Kolmogorov--Doob submartingale inequality and using
Lem\-ma~\ref{lm:6.6} (with $k=3$), we obtain
\[
\QQ_{z}^{r}\Bigl\{{\sup_{0\le
t\le\infty}}|\xi_{0j}(t)|>\varepsilon n_j
\Bigr\}\le
\frac{\EE_{z}^{r}[\xi_j-\EE_z^r(\xi_j)
]^6}{(\varepsilon
n_j)^{6}} =O(|n|^{-2}).
\]
On the other hand, by Corollary \ref{cor:Q}
\[
\QQ_{z}^{r}\{\xi=n\}\asymp(n_1n_2)^{-2/3}\asymp|n|^{-4/3}.
\]
In view of these estimates, the right-hand side of (\ref{A1}) is
dominated by a~quantity of order of $O(|n|^{-2/3})\to0$,
and the theorem is proved.
\end{pf}

\begin{appendix}\label{app}

\section*{Appendix}

\subsection{Tangential distance between convex paths}\label{A-1}
Let ${\mathcal G}_0$ be the space of paths in $\RR^2_+$ starting
from the origin and such that each path $\gamma\in{\mathcal G}_0$ is
continuous, piecewise $C^1$-smooth (i.e., everywhere except a finite
set), bounded and convex and, furthermore, its tangent slope (where
it exists) is nonnegative, including the possible value $+\infty$.
Convexity implies that the slope is nondecreasing as a function of
the natural parameter (i.e., the length along the path measured from
the origin).

For $\gamma\in{\mathcal G}_0$, let $g_\gamma(t)
=(g_{1}(t),g_{2}(t))$ denote the right endpoint of the (closure of
the) part of $\gamma$ where the tangent slope does not exceed
$t\in[0,\infty]$. Note that the functions $x_{1}=g_{1}(t)$,
$x_2=g_{2}(t)$ are c\`{a}dl\`{a}g (i.e., right-continuous with left
limits), and
%
\setcounter{equation}{0}
\begin{equation}\label{eq:slope}
\frac{\dif x_2}{\dif x_1}=\frac{g^{\prime
}_2(t)}{g^{\prime}_1(t)}=t.
\end{equation}
More precisely, equation (\ref{eq:slope}) holds at points where the
tangent exists and its slope is strictly growing; corners on
$\gamma$ correspond to intervals where both functions $g_1$ and
$g_2$ are constant, whereas flat (straight line) pieces on $\gamma$
lead to simultaneous jumps of $g_1$ and $g_2$.

The canonical limit shape curve $\gamma^*$ [see (\ref{eq:gamma0})]
is determined by the parametric equations [cf. (\ref{eq:g*}),
(\ref{u_12})]
%
\begin{equation}\label{u_12_}\quad
x_1=g^*_{1}(t)=\frac{t^2+2t}{(1+t)^2} ,\qquad
x_2=g^*_{2}(t)=\frac{t^2}{(1+t)^2} ,\qquad   t\in
[0,\infty].
\end{equation}
Indeed, it can be readily seen that the functions (\ref{u_12_})
satisfy the Cartesian equation (\ref{eq:gamma0}) for $\gamma^*$;
moreover, it is easy to check that $t$ in equations (\ref{u_12_}) is
the tangential parameter
\[
\frac{\dif g^{*}_{1}(t)}{\dif t}= \frac{2}{(1+t)^3}
,\qquad
\frac{\dif g^{*}_{2}}{\dif t}= \frac{2t}{(1+t)^3} ,
\]
and hence [cf. (\ref{eq:slope})]
\[
\frac{\dif x_2}{\dif x_1}=\frac{\dif g^{*}_{2}/\dif t}{\dif
 g^{*}_{1}/\dif t}=t.
\]

The tangential distance $d_{\mathcal T}$ between paths in ${\mathcal
G}_0$ is defined as follows:
%
\begin{equation}\label{eq:dT}
d_{\mathcal T}(\gamma_1,\gamma_2):={\sup_{0\le t\le\infty}} |
g_{\gamma_1}(t)-g_{\gamma_2}(t)|,\qquad
\gamma_1,\gamma_2\in{\mathcal G}_0.
\end{equation}
\setcounter{theorem}{0}
\begin{lemma}\label{lm:A1}
The Hausdorff distance $d_{\mathcal H}$ defined in
(\ref{eq:dH}) is dominated by the tangential distance
$d_{\mathcal T}$
%
\begin{equation}\label{eq:dHdT}
d_{\mathcal H}(\gamma_1,\gamma_2)\le d_{\mathcal T}
(\gamma_1,\gamma_2),\qquad \gamma_1,\gamma_2\in{\mathcal G}_0.
\end{equation}
\end{lemma}
\begin{pf}
First of all, note that any path $\gamma\in{\mathcal G}_0$ can be
approximated, simultaneously in metrics $d_{\mathcal H}$ and
$d_{\mathcal T}$, by polygonal lines $\gamma^m$ (e.g., by
inscribing polygonal lines with refined edges in the arc $\gamma$)
so that
\[
\lim_{m\to\infty} d_{\mathcal H}(\gamma,\gamma^m)=0,\qquad
\lim_{m\to\infty} d_{\mathcal T}(\gamma,\gamma^m)=0.
\]
This reduces inequality (\ref{eq:dHdT}) to the case where
$\gamma_1$, $\gamma_2$ are polygonal lines. Moreover, by symmetry it
suffices to show that
%
\begin{equation}\label{eq:t*1a}
{\max_{x\in\gamma_1}\min_{y\in\gamma_2}}|x-y|\le d_{\mathcal
T}(\gamma_1,\gamma_2).
\end{equation}

Note that if a point $x\in\gamma_1$ can be represented as
$x=g_{\gamma_1}(t_0)$ with some $t_0$ (i.e., $x$ is a vertex of
$\gamma_1$), then
\[
{\min_{y\in
\gamma_2}}|x-y|={\min_{y\in\gamma_2}}|g_{\gamma_1}(t_0)-y|\le
|g_{\gamma_1}(t_0)-g_{\gamma_2}(t_0)|\le d_{\mathcal
T}(\gamma_1,\gamma_2),
\]
and inequality (\ref{eq:t*1a}) follows.

Suppose now that $x\in\gamma_1$ lies on an edge---say of slope
$t^*$---between two consecutive vertices $g_{\gamma_1}(t_*)$
and $g_{\gamma_1}(t^*)$, then $x=\theta
g_{\gamma_1}(t_*)+(1-\theta)g_{\gamma_1}(t^*)$ with some
$\theta\in(0,1)$ and
%
\begin{eqnarray}\label{eq:t**}
{\min_{y\in\gamma_2}}|x-y|&=& {\min_{y\in
\gamma_2}}|\theta g_{\gamma
_1}(t_*)+(1-\theta)g_{\gamma_1}(t^*)-y|\nonumber\\
&\le&|\theta g_{\gamma_1}(t_*)+(1-\theta
)g_{\gamma_1}(t^*)-g_{\gamma_2}(t^*)|\nonumber\\[-8pt]\\[-8pt]
&\le&\theta|g_{\gamma_1}(t_*)-g_{\gamma_2}(t^*)|+
(1-\theta)|g_{\gamma_1}(t^*)-g_{\gamma_2}(t^*)|\nonumber\\
&\le&
\theta|g_{\gamma_1}(t_*)-g_{\gamma_2}(t^*)|+(1-\theta
)
d_{\mathcal T}(\gamma_1,\gamma_2).\nonumber
\end{eqnarray}
Note that for all $t\in[t_*,t^*)$ we have
$g_{\gamma_1}(t)=g_{\gamma_1}(t_*)$, hence
%
\begin{equation}\label{eq:t*0}\qquad
|g_{\gamma_1}(t_*)-g_{\gamma_2}(t)|=|g_{\gamma_1}(t)-g_{\gamma
_2}(t)|\le
d_{\mathcal T}(\gamma_1,\gamma_2)\qquad (t_*\le t< t^*).
\end{equation}

If $g_{\gamma_2}(t)$ is continuous at $t=t^*$, then inequality
(\ref{eq:t*0}) extends to $t=t^*$
\[
|g_{\gamma_1}(t_*)-g_{\gamma_2}(t^*)|\le d_{\mathcal
T}(\gamma_1,\gamma_2).
\]
Substituting this inequality into the right-hand side of
(\ref{eq:t**}), we see that ${\min_{y\in\gamma_2}}|x-y|\le
d_{\mathcal T}(\gamma_1,\gamma_2)$, which implies (\ref{eq:t*1a}).

If $g_{\gamma_2}(t^*-0)\ne g_{\gamma_2}(t^*)$, then $t^*$ coincides
with the slope of some edge on~$\gamma_2$ [with the endpoints, say,
$g_{\gamma_2}(t'_*)$ and $g_{\gamma_2}(t^*)$], which is thus
parallel to the edge on $\gamma_1$ where the point $x$ lies [i.e.,
with the endpoints $g_{\gamma_1}(t_*)$, $g_{\gamma_1}(t^*)$].
Setting $s_*:=\max\{t'_*,t_*\}<t^*$, we have
$g_{\gamma_1}(t_*)=g_{\gamma_1}(s_*)$, $g_{\gamma_2}(t'_*)=
g_{\gamma_2}(s_*)$.

To complete the proof, it remains to observe that the shortest
distance from a point on a base of a trapezoid to the opposite base
does not exceed the maximum length of the two lateral sides. Hence,
\begin{eqnarray*}
{\min_{y\in\gamma_2}}|x-y| &\le&\min\{|x-y|\dvtx y\in
[g_{\gamma_2}(s_*),g_{\gamma_2}(t^*)]\}\\
&\le&\max\{|g_{\gamma_1}(s_*)-g_{\gamma_2}(s_*)|,
|g_{\gamma_1}(t^*)-g_{\gamma_2}(t^*)|\}\\
&\le& d_{\mathcal T}(\gamma_1,\gamma_2),
\end{eqnarray*}
and the bound (\ref{eq:t*1a}) follows.
\end{pf}
\begin{remark}
Note, however, that the metrics $d_{\mathcal H}$ and $d_{\mathcal
T}$ are \textit{not equivalent}. For instance, if
$\gamma\in{\mathcal G}_0$ is a smooth, strictly convex curve with
curvature bounded below by a constant $\varkappa_0>0$, then for an
inscribed polygonal line~$\varGamma_\varepsilon$ with edges of
length no more than $\varepsilon>0$, its tangential distance from
$\gamma$ is of the order of $\varepsilon$, but the Hausdorff
distance is of the order of $\varepsilon^2$
\[
d_{\mathcal T}(\varGamma_\varepsilon,\gamma)\asymp\varepsilon
,\qquad
d_{\mathcal H}
(\varGamma_\varepsilon,\gamma)\asymp\varkappa_0 \varepsilon^2
\qquad (\varepsilon\to0).
\]
Moreover, in the degenerate case where the curvature may vanish, the
difference between the two metrics may be even more dramatic. For
instance, it is possible that two polygonal lines are close to each
other in the Hausdorff distance while their tangential distance is
quite large. For an example, consider two straight line segments
$\varGamma_1, \varGamma_2\subset\RR_+^2$ of the same (large) length~$L$,
both starting from the origin and with very close slopes, so
that the Euclidean distance $\delta$ between their right endpoints
is small; then $d_{\mathcal H}(\varGamma_1,\varGamma_2)\le\delta$
whereas $d_{\mathcal T}(\varGamma_1,\varGamma_2)=L$.
\end{remark}

\subsection{Total variation distance between $\PP_n^r$ and $\PP_n^1$}\label{A-2}

Note that if probability measures $\tilde\PP_n$ and $\PP_n$ on the
polygonal space $\CP_n$ are asymptotically close to each other in
\textit{total variation} (\textit{TV}), that is, $\|\tilde
\PP_n-\PP_n\|_{\mathrm{TV}}\to0$ as $n\to\infty$, where
\[
\|\tilde\PP_n-\PP_n\|_{\mathrm{TV}}:={\sup_{A\subset\CP_n}}|\tilde
\PP_n(A)-\PP_n(A)|,
\]
then the problem of universality of the limit shape is resolved in a
trivial way, in that if a limit shape $\gamma^*$ exists under
$\PP_n$ then the\vspace*{1pt} same curve $\gamma^*$ provides the limit shape
under $\tilde\PP_n$. Indeed, assuming that the event
$A_\varepsilon=\{d(\tilde\varGamma_n,\gamma^*)>\varepsilon\}$
satisfies $\PP_n(A_\varepsilon)\to0$ as $n\to\infty$, we have
\begin{eqnarray*}
\tilde\PP_n(A_\varepsilon)&\le&\PP_n (A_\varepsilon) +
|\tilde\PP_n(A_\varepsilon)-\PP_n(A_\varepsilon)|\\
&\le&
\PP_n (A_\varepsilon) + \|\tilde\PP_n -\PP_n\|_{\mathrm{TV}}\to0\qquad
(n\to\infty).
\end{eqnarray*}

However, the family $\{\PP_n^r\}$, defined by formula (\ref{condP})
with the coefficients~(\ref{b_k}), is \textit{not} close to $\PP^1_n$
in total variation, at least uniformly in $r\in(0,\infty)$.
\begin{theorem}\label{th:Vasser}
For every fixed  $n$, the limiting distance in total
variation between  $\PP_n^r$ and  $\PP_n^1$, as  $r\to
0$ or  $r\to\infty$, is given by
%
\begin{equation}\label{eq:limVar}
{\lim_{r\to0,\infty}}
\|\PP_n^r-\PP_n^1\|_{\mathrm{TV}}=1-\frac{1}{\#(\CP_n)} .
\end{equation}
\end{theorem}

\begin{pf}
To obtain a lower bound for $\|\PP_n^r-\PP_n^1\|_{\mathrm{TV}}$ in the case
$r\to\infty$, consider the polygonal line $\varGamma^* \in\CP_n$
consisting of two edges, horizontal [from the origin to $(n_1,0)$]
and vertical [from $(n_1,0)$ to $n=(n_1,n_2)$]. The corresponding
configuration $\nu_{\varGamma^*}$ is determined by the conditions
$\nu_{\varGamma^*}(1,0)=n_1$,  $\nu_{\varGamma^*}(0,1)=n_2$ and
$\nu_{\varGamma^*}(x)=0$ otherwise. Note that $b_k^{r}\sim r^k/k!$
as $r\to\infty$ [see (\ref{b_k})], hence
$b^r (\varGamma)=O(r^{N_\varGamma})$, where
$N_\varGamma:=\sum_{x\in\calX}\nu_{\varGamma}(x)$ is the total
number of integer points on $\varGamma$ (excluding the origin). We
have $N_{\varGamma^*}=n_1+n_2$, so $b^r(\varGamma^*)=b_{n_1}^r
b_{n_2}^r\sim r^{n_1+n_2}/(n_1! n_2!)$ as $r\to\infty$.
On the
other hand, for any $\varGamma\in\CP_n$ ($\varGamma\ne\varGamma^*$)
one has $b^r(\varGamma)=o(r^{n_1+n_2})$. Indeed, for $x\in\calX$ we
have $x_1+x_2\ge1$ and, moreover, $x_1+x_2>1$ unless $x=(1,0)$ or
$x=(0,1)$. Hence, for any $\varGamma\in\CP_n$
($\varGamma\ne\varGamma^*$),
\[
N_\varGamma=\sum_{x\in\calX} \nu_{\varGamma}(x)<\sum_{x\in\calX}
(x_1+x_2) \nu_{\varGamma}(x)= n_1+n_2,
\]
so that $N_\varGamma<n_1+n_2$ and
$b^r(\varGamma)=O(r^{N_\varGamma})=o(r^{n_1+n_2})$ as $r\to\infty$.
Therefore, from (\ref{condP}) we get
\[
\PP_n^r(\varGamma^*)=\frac{b^r({\varGamma^*})}{b^r({\varGamma^*})
+\sum_{\varGamma\ne\varGamma^*}b^r({\varGamma})}=\frac{r^{n_1+n_2}}
{r^{n_1+n_2}+ o(r^{n_1+n_2})}\to1\qquad (r\to\infty),
\]
and it follows that
%
\begin{equation}\label{b*1}\quad
\|\PP_n^r-\PP_n^1\|_{\mathrm{TV}}
\ge|\PP_n^r(\varGamma^*)-\PP_n^1(\varGamma^*)|\to
1-\frac{1}{\#(\CP_n)}\qquad (r\to\infty).
\end{equation}

For the case $r\to0$, consider the polygonal line
$\varGamma_* \in\CP_n$ consisting of one edge leading
from the
origin to $n=(n_1,n_2)$. That is, $\nu_{\varGamma_*} (n/k_n)=k_n$
and $\nu(x)=0$ otherwise, where $k_n:=\gcd(n_1,n_2)$. Clearly,
$b^r({\varGamma_*})=b^r_{k_n}=r/k_n$, while for any other
$\varGamma\in\CP_n$ (i.e., with more than one edge), by (\ref{b_k})
we have $b^r({\varGamma})=O(r^2)$ as $r\to0$. Therefore, according
to (\ref{condP}),
\[
\PP_n^r(\varGamma_*)= \frac{b^r({\varGamma_*})}{b^r({\varGamma_*})
+\sum_{\varGamma\ne\varGamma_*}b^r({\varGamma})}= \frac{r}
{r+O(r^2)}\to1\qquad (r\to0),
\]
and similarly to (\ref{b*1}) we obtain
%
\begin{equation}\label{b*2}
\|\PP_n^r-\PP_n^1\|_{\mathrm{TV}}
\ge|\PP_n^r(\varGamma_*)-\PP_n^1(\varGamma_*)|\to
1-\frac{1}{\#(\CP_n)}\qquad (r\to0).\hspace*{-30pt}
\end{equation}

The upper bound for $\|\PP_n^r-\PP_n^1\|_{\mathrm{TV}}$ (uniform in $r$)
follows from the known fact (see \cite{D}, page 472,
and also \cite{ABT}, Section 3.1, pages 67 and 68)
that the total
variation distance can be expressed in terms of a certain
Vasershtein (--Kantorovich--Rubinstein, cf. \cite{VKantorovich})
distance
\[
\|\PP_n^r-\PP_n^1\|_{\mathrm{TV}}=\inf_{X,Y}\sfE [\varrho(X,Y)].
\]
Here the infimum is taken over all pairs of random elements $X$ and
$Y$ defined on a common probability space $(\Omega,{\mathcal
F},\sfP)$ with values in $\CP_n$ and the marginal distributions
$\PP_n^r$ and $\PP_n^1$, respectively; $\sfE$ denotes expectation
with respect to the probability measure $\sfP$, and the function
$\varrho(\cdot,\cdot)$ on $\CP_n\times\CP_n$ is such that
\mbox{$\varrho(\varGamma,\varGamma')=1$} if $\varGamma\ne\varGamma'$ and
$\varrho(\varGamma,\varGamma')=0$ if $\varGamma=\varGamma'$
(therefore defining a~discrete metric in $\CP_n$). Choosing $X$ and
$Y$ so that they are independent of each other, we obtain
\begin{eqnarray*}
\|\PP_n^r-\PP_n^1\|_{\mathrm{TV}}&\le&\sfE[\varrho(X,Y)]=
1-\sfP\{X=Y\}\\
&=&1-\sum_{\varGamma\in\CP_n}\sfP\{X=\varGamma,Y=\varGamma\}
=1-\sum_{\varGamma\in\CP_n}\PP_n^r(\varGamma)\cdot\PP
_n^1(\varGamma)\\
&=&1-\frac{1}{\#(\CP_n)}\sum_{\varGamma\in\CP_n}\PP_n^r(\varGamma)
=1-\frac{1}{\#(\CP_n)} .
\end{eqnarray*}
Combining this estimate with (\ref{b*1}) and (\ref{b*2}), we obtain
(\ref{eq:limVar}).
\end{pf}

In the limit $n\to\infty$, Theorem \ref{th:Vasser} yields
\[
{\lim_{n\to\infty}\lim_{r\to0,\infty}}
\|\PP_n^r-\PP_n^1\|_{\mathrm{TV}}=1.
\]
That is to say, in the successive limit $r\to0 (\infty)$,
$n\to\infty$, the measures $\PP_n^r$ and $\PP_n^1$ become singular
with respect to each other.

Moreover, one can show that the distance $\|\PP_n^r-\PP_n^1\|_{\mathrm{TV}}$
is not small even for a fixed $r\ne1$. To this end, it suffices to
find a function on $\CP_n$ possessing a limiting distribution
(possibly degenerate) under each $\PP_n^r$, with the limit depending
on the parameter $r$. Recalling Remark \ref{rm:r}, it is natural to
seek such a function in the form referring to integer points on
$\varGamma\in\CP_n$. Indeed, for the statistic
$N_\varGamma=\sum_{x\in\calX}\nu_{\varGamma}(x)$ introduced in the
proof of Theorem \ref{th:Vasser}, the following law of large numbers
holds (see Bogachev and Zarbaliev \cite{BZ1}, Theorem 3, and
Zarbaliev \cite{Z}, Section 1.10).
\begin{lemma}\label{lm:A3}
Under Assumption  \ref{as:c}, for each
$r\in(0,\infty)$ and any  $\varepsilon>0$,
%
\begin{equation}\label{eq:A}
\lim_{n\to\infty} \PP_n^r(A_\varepsilon^r)=1\qquad
\mbox{where }
A_\varepsilon^r:=\biggl\{\biggl|\frac{N_\varGamma}{(n_1 n_2)^{1/3}}-
\frac{r^{1/3}}{\kappa^2}\biggr|\le\varepsilon\biggr\}.
\end{equation}
\end{lemma}

The result (\ref{eq:A}) implies that, for any $r\ne1$ and all
$\varepsilon>0$ small enough,
\[
\|\PP_n^r-\PP_n^1\|_{\mathrm{TV}}\ge
|\PP_n^r(A_\varepsilon^r)-\PP_n^1(A_\varepsilon^r)|\quad\to\quad
|1-0|=1 \qquad (n\to\infty),
\]
and we arrive at the following result.
\begin{theorem}\label{th:lim(r)}
For every fixed  $r\ne1$, we have
\[
{\lim_{n\to\infty}}\|\PP_n^r-\PP_n^1\|_{\mathrm{TV}}=1,
\]
and hence the measures  $\PP_n^r$ and $\PP_n^1$ on $\CP_n$ are
asymptotically singular with respect to each other as
$n\to\infty$.
\end{theorem}
\end{appendix}

\section*{Acknowledgments}
We would like to express our deep gratitude to Ya. G. Sinai for
introducing us to the beautiful area of random partitions and for
stimulating discussions over the years. Our sincere thanks are also
due to Yu. V. Prokhorov and A. M. Vershik for useful remarks and
discussions at various stages of this work. We are also thankful to
the anonymous referees for the careful reading of the manuscript
and useful suggestions that have helped to improve the presentation
of the paper.

%

%
\printaddresses


\begin{thebibliography}{41}

\bibitem{ABT}
%
\begin{bbook}[vtex]
\bauthor{\bsnm{Arratia},~\bfnm{Richard}\binits{R.}},
\bauthor{\bsnm{Barbour},~\bfnm{A.~D.}\binits{A.~D.}} \AND
\bauthor{\bsnm{Tavar{\'e}},~\bfnm{Simon}\binits{S.}}
(\byear{2003}).
\btitle{Logarithmic Combinatorial Structures: A Probabilistic Approach}.
\bpublisher{European Mathematical Society},
\baddress{Z\"urich}.
\bid{doi={10.4171/000}, mr={2032426}}
\end{bbook}
%
\endbibitem

\bibitem{AT}
%
\begin{barticle}[mr]
\bauthor{\bsnm{Arratia},~\bfnm{Richard}\binits{R.}} \AND
\bauthor{\bsnm{Tavar{\'e}},~\bfnm{Simon}\binits{S.}}
(\byear{1994}).
\btitle{Independent process approximations for random combinatorial
structures}.
\bjournal{Adv. Math.}
\bvolume{104}
\bpages{90--154}.
\bid{doi={10.1006/aima.1994.1022}, mr={1272071}}
\end{barticle}
%
\endbibitem

\bibitem{B}
%
\begin{barticle}[mr]
\bauthor{\bsnm{B{\'a}r{\'a}ny},~\bfnm{I.}\binits{I.}}
(\byear{1995}).
\btitle{The limit shape of convex lattice polygons}.
\bjournal{Discrete Comput. Geom.}
\bvolume{13}
\bpages{279--295}.
\bid{doi={10.1007/BF02574045}, mr={1318778}}
\end{barticle}
%
\endbibitem

\bibitem{Bellman}
%
\begin{bbook}[vtex]
\bauthor{\bsnm{Bellman},~\bfnm{Richard}\binits{R.}}
(\byear{1970}).
\btitle{Introduction to Matrix Analysis},
\bedition{2nd} ed.
\bpublisher{McGraw-Hill},
\baddress{New York}.
\bid{mr={0258847}}
\end{bbook}
%
\endbibitem

\bibitem{BR}
%
\begin{bbook}[vtex]
\bauthor{\bsnm{Bhattacharya},~\bfnm{R.~N.}\binits{R.~N.}} \AND
\bauthor{\bsnm{Ranga~Rao},~\bfnm{R.}\binits{R.}}
(\byear{1986}).
\btitle{Normal Approximation and Asymptotic Expansions}.
\bpublisher{Krieger}, \baddress{Malabar, FL}.
\bid{mr={0855460}}
\end{bbook}
%
\endbibitem

\bibitem{BZ1}
%
\begin{barticle}[mr]
\bauthor{\bsnm{Bogachev},~\bfnm{L.~V.}\binits{L.~V.}} \AND
\bauthor{\bsnm{Zarbaliev},~\bfnm{S.~M.}\binits{S.~M.}}
(\byear{1999}).
\btitle{Limit theorems for a certain class of random convex polygonal lines}.
\bjournal{Russian Math. Surveys}
\bvolume{54}
\bpages{830--832}.
\end{barticle}
%
\endbibitem

\bibitem{BZ2}
%
\begin{barticle}[vtex]
\bauthor{\bsnm{Bogachev},~\bfnm{L.~V.}\binits{L.~V.}} \AND
\bauthor{\bsnm{Zarbaliev},~\bfnm{S.~M.}\binits{S.~M.}}
(\byear{1999}).
\btitle{Approximation of convex functions by random polygonal lines}.
\bjournal{Dokl. Math.}
\bvolume{59}
\bpages{46--49}.
\end{barticle}
%
\endbibitem

\bibitem{BZ2a}
%
\begin{bmisc}[vtex]
\bauthor{\bsnm{Bogachev},~\bfnm{L.~V.}\binits{L.~V.}} \AND
\bauthor{\bsnm{Zarbaliev},~\bfnm{S.~M.}\binits{S.~M.}}
(\byear{2004}).
\bhowpublished{Approximation of convex curves by random lattice polygons.
Preprint, NI04003-IGS, Isaac Newton Inst. Math. Sci., Cambridge. Available
at}
\url{http://www.newton.cam.ac.uk/preprints/NI04003.pdf}.
\end{bmisc}
%
\endbibitem

\bibitem{BZ3}
%
\begin{barticle}[vtex]
\bauthor{\bsnm{Bogachev},~\bfnm{L.~V.}\binits{L.~V.}} \AND
\bauthor{\bsnm{Zarbaliev},~\bfnm{S.~M.}\binits{S.~M.}}
(\byear{2009}).
\btitle{A proof of the Vershik--Prohorov conjecture on the universality of the
limit shape for a class of random polygonal lines}.
\bjournal{Dokl. Math.}
\bvolume{79}
\bpages{197--202}.
\end{barticle}
%
\endbibitem

\bibitem{Comtet}
%
\begin{barticle}[mr]
\bauthor{\bsnm{Comtet},~\bfnm{Alain}\binits{A.}},
\bauthor{\bsnm{Majumdar},~\bfnm{Satya~N.}\binits{S.~N.}},
\bauthor{\bsnm{Ouvry},~\bfnm{St{\'e}phane}\binits{S.}} \AND
\bauthor{\bsnm{Sabhapandit},~\bfnm{Sanjib}\binits{S.}}
(\byear{2007}).
\btitle{Integer partitions and exclusion statistics: Limit shapes and the
largest parts of {Y}oung diagrams}.
\bjournal{J.~Stat. Mech. Theory Exp.}
\bpages{P10001 (electronic)}.
\bid{mr={2358050}}
\end{barticle}
%
\endbibitem

\bibitem{D}
%
\begin{barticle}[vtex]
\bauthor{\bsnm{Dobrushin},~\bfnm{R.~L.}\binits{R.~L.}}
(\byear{1970}).
\btitle{Prescribing a system of random
variables by conditional distributions}.
\bjournal{Theory Probab. Appl.}
\bvolume{15}
\bpages{458--486}.
\end{barticle}
%
\endbibitem

\bibitem{F}
%
\begin{bbook}[vtex]
\bauthor{\bsnm{Feller},~\bfnm{William}\binits{W.}}
(\byear{1968}).
\btitle{An Introduction to Probability Theory and Its Applications. {V}ol.
{I}},
\bedition{3rd} ed.
\bpublisher{Wiley}, \baddress{New York}.
\bid{mr={0228020}}
\end{bbook}
%
\endbibitem

\bibitem{FG}
%
\begin{barticle}[mr]
\bauthor{\bsnm{Freiman},~\bfnm{Gregory~A.}\binits{G.~A.}} \AND
\bauthor{\bsnm{Granovsky},~\bfnm{Boris~L.}\binits{B.~L.}}
(\byear{2005}).
\btitle{Clustering in coagulation-fragmentation processes, random combinatorial
structures and additive number systems: Asymptotic formulae and limiting
laws}.
\bjournal{Trans. Amer. Math. Soc.}
\bvolume{357}
\bpages{2483--2507}.
\bid{doi={10.1090/S0002-9947-04-03617-7}, mr={2140447}}
\end{barticle}
%
\endbibitem

\bibitem{FVY}
%
\begin{barticle}[vtex]
\bauthor{\bsnm{Freiman},~\bfnm{G.~A.}\binits{G.~A.}},
\bauthor{\bsnm{Vershik},~\bfnm{A.~M.}\binits{A.~M.}} \AND
\bauthor{\bsnm{Yakubovich},~\bfnm{Yu.~V.}\binits{Y.~V.}}
(\byear{2000}).
\btitle{A local limit theorem for random strict partitions}.
\bjournal{Theory Probab. Appl.}
\bvolume{44}
\bpages{453--468}.
\bptnote{check year}
\end{barticle}
%
\endbibitem

\bibitem{FY}
%
\begin{bincollection}[mr]
\bauthor{\bsnm{Freiman},~\bfnm{Gregory~A.}\binits{G.~A.}} \AND
\bauthor{\bsnm{Yudin},~\bfnm{Alexander~A.}\binits{A.~A.}}
(\byear{2006}).
\btitle{The interface between probability theory and additive number theory
(local limit theorems and structure theory of set addition)}.
In \bbooktitle{Representation Theory, Dynamical Systems, and Asymptotic
Combinatorics}
(\beditor{\bfnm{V.}\binits{V.}~\bsnm{Kaimanovich}} \AND
  \beditor{\bfnm{A.}\binits{A.}~\bsnm{Lodkin}}, eds.).
\bseries{Amer. Math. Soc. Transl. Ser. 2}
\bvolume{217}
\bpages{51--72}.
\bpublisher{Amer. Math. Soc.}, \baddress{Providence, RI}.
\bid{mr={2276101}}
\end{bincollection}
%
\endbibitem

\bibitem{Fristedt}
%
\begin{barticle}[mr]
\bauthor{\bsnm{Fristedt},~\bfnm{Bert}\binits{B.}}
(\byear{1993}).
\btitle{The structure of random partitions of large integers}.
\bjournal{Trans. Amer. Math. Soc.}
\bvolume{337}
\bpages{703--735}.
\bid{doi={10.2307/2154239}, mr={1094553}}
\end{barticle}
%
\endbibitem

\bibitem{Halmos}
%
\begin{bbook}[vtex]
\bauthor{\bsnm{Halmos},~\bfnm{Paul~R.}\binits{P.~R.}}
(\byear{1951}).
\btitle{Introduction to {H}ilbert Space and the Theory of Spectral
Multiplicity}.
\bpublisher{Chelsea}, \baddress{New York}.
\bid{mr={0045309}}
\bptnote{check year}
\end{bbook}
%
\endbibitem

\bibitem{HW}
%
\begin{bbook}[vtex]
\bauthor{\bsnm{Hardy},~\bfnm{G.~H.}\binits{G.~H.}} \AND
\bauthor{\bsnm{Wright},~\bfnm{E.~M.}\binits{E.~M.}}
(\byear{1960}).
\btitle{An Introduction to the Theory of Numbers},
\bedition{4th} ed.
\bpublisher{Oxford Univ. Press}, \baddress{Oxford}.
\end{bbook}
%
\endbibitem

\bibitem{Horn}
%
\begin{bbook}[mr]
\bauthor{\bsnm{Horn},~\bfnm{Roger~A.}\binits{R.~A.}} \AND
\bauthor{\bsnm{Johnson},~\bfnm{Charles~R.}\binits{C.~R.}}
(\byear{1990}).
\btitle{Matrix Analysis}.
\bpublisher{Cambridge Univ. Press}, \baddress{Cambridge}.
\bid{mr={1084815}}
\end{bbook}
%
\endbibitem

\bibitem{Iv}
%
\begin{bbook}[vtex]
\bauthor{\bsnm{Ivi{\'c}},~\bfnm{Aleksandar}\binits{A.}}
(\byear{1985}).
\btitle{The {R}iemann Zeta-Function: The Theory of the Riemann Zeta-Function with Applications}.
\bpublisher{Wiley}, \baddress{New York}.
\bid{mr={0792089}}
\end{bbook}
%
\endbibitem

\bibitem{Ka}
%
\begin{bbook}[vtex]
\bauthor{\bsnm{Karatsuba},~\bfnm{Anatolij~A.}\binits{A.~A.}}
(\byear{1993}).
\btitle{Basic Analytic Number Theory}.
\bpublisher{Springer}, \baddress{Berlin}.
\bid{mr={1215269}}
\end{bbook}
%
\endbibitem

\bibitem{Kh1}
%
\begin{bbook}[vtex]
\bauthor{\bsnm{Khinchin},~\bfnm{A.~I.}\binits{A.~I.}}
(\byear{1949}).
\btitle{Mathematical {F}oundations of {S}tatistical {M}echanics}.
\bpublisher{Dover}, \baddress{New York}.
\bid{mr={0029808}}
\end{bbook}
%
\endbibitem

\bibitem{Kh2}
%
\begin{bbook}[vtex]
\bauthor{\bsnm{Khinchin},~\bfnm{A.~Y.}\binits{A.~Y.}}
(\byear{1960}).
\btitle{Mathematical Foundations of Quantum Statistics}.
\bpublisher{Graylock Press}, \baddress{Albany, NY}.
\bid{mr={0111217}}
\end{bbook}
%
\endbibitem

\bibitem{Kolchin}
%
\begin{bbook}[mr]
\bauthor{\bsnm{Kolchin},~\bfnm{V.~F.}\binits{V.~F.}}
(\byear{1999}).
\btitle{Random Graphs}.
\bseries{Encyclopedia of Mathematics and Its Applications}
\bvolume{53}.
\bpublisher{Cambridge Univ. Press}, \baddress{Cambridge}.
\bid{mr={1728076}}
\end{bbook}
%
\endbibitem

\bibitem{Lancaster}
%
\begin{bbook}[mr]
\bauthor{\bsnm{Lancaster},~\bfnm{Peter}\binits{P.}}
(\byear{1969}).
\btitle{Theory of Matrices}.
\bpublisher{Academic Press}, \baddress{New York}.
\bid{mr={0245579}}
\end{bbook}
%
\endbibitem

\bibitem{Pitman}
%
\begin{bbook}[vtex]
\bauthor{\bsnm{Pitman},~\bfnm{J.}\binits{J.}}
(\byear{2006}).
\btitle{Combinatorial Stochastic Processes}.
\bseries{Lecture Notes in Math.}
\bvolume{1875}.
\bpublisher{Springer}, \baddress{Berlin}.
\bid{mr={2245368}}
\end{bbook}
%
\endbibitem

\bibitem{Prokhorov}
%
\begin{bmisc}[vtex]
\bauthor{\bsnm{Prokhorov},~\bfnm{Yu.~V.}\binits{Yu.~V.}}
(\byear{1998}).
\bhowpublished{Private communication}.
\end{bmisc}
%
\endbibitem

\bibitem{Ruelle}
%
\begin{bbook}[vtex]
\bauthor{\bsnm{Ruelle},~\bfnm{David}\binits{D.}}
(\byear{1969}).
\btitle{Statistical Mechanics: {R}igorous Results}.
\bpublisher{Benjamin}, \baddress{New York}.
\bid{mr={0289084}}
\end{bbook}
%
\endbibitem

\bibitem{S}
%
\begin{barticle}[mr]
\bauthor{\bsnm{Sinai},~\bfnm{Ya.~G.}\binits{Ya.~G.}}
(\byear{1994}).
\btitle{Probabilistic approach to the
analysis of statistics for convex polygonal lines}.
\bjournal{Funct. Anal. Appl.}
\bvolume{28}
\bpages{108--113}.
\end{barticle}
%
\endbibitem

\bibitem{Titch2}
%
\begin{bbook}[vtex]
\bauthor{\bsnm{Titchmarsh},~\bfnm{E.~C.}\binits{E.~C.}}
(\byear{1952}).
\btitle{The Theory of Functions},
\bedition{2nd} ed.
\bpublisher{Oxford Univ. Press}, \baddress{Oxford}.
\end{bbook}
%
\endbibitem

\bibitem{Titch1}
%
\begin{bbook}[mr]
\bauthor{\bsnm{Titchmarsh},~\bfnm{E.~C.}\binits{E.~C.}}
(\byear{1986}).
\btitle{The Theory of the {R}iemann Zeta-Function}, \bedition{2nd} ed.
\bpublisher{Oxford Univ. Press}, \baddress{Oxford}.
\bid{mr={0882550}}
\end{bbook}
%
\endbibitem

\bibitem{V1}
%
\begin{barticle}[mr]
\bauthor{\bsnm{Vershik},~\bfnm{A.~M.}\binits{A.~M.}}
(\byear{1994}).
\btitle{The limit form of convex integral polygons and related problems}.
\bjournal{Funct. Anal. Appl.}
\bvolume{28}
\bpages{13--20}.
\end{barticle}
%
\endbibitem

\bibitem{V2}
%
\begin{binproceedings}[mr]
\bauthor{\bsnm{Vershik},~\bfnm{Anatoly~M.}\binits{A.~M.}}
(\byear{1995}).
\btitle{Asymptotic combinatorics and algebraic analysis}.
In \bbooktitle{Proceedings of the {I}nternational {C}ongress of
{M}athematicians ({Z}\"urich, 1994)}
\bvolume{2}
\bpages{1384--1394}.
\bpublisher{Birkh\"auser}, \baddress{Basel}.
\bid{mr={1404040}}
\end{binproceedings}
%
\endbibitem

\bibitem{V3}
%
\begin{barticle}[mr]
\bauthor{\bsnm{Vershik},~\bfnm{A.~M.}\binits{A.~M.}}
(\byear{1996}).
\btitle{Statistical mechanics of combinatorial partitions, and their limit
configurations}.
\bjournal{Funct. Anal. Appl.}
\bvolume{30}
\bpages{90--105}.
\end{barticle}
%
\endbibitem

\bibitem{V4}
%
\begin{barticle}[vtex]
\bauthor{\bsnm{Vershik},~\bfnm{A.~M.}\binits{A.~M.}}
(\byear{1997}).
\btitle{Limit distribution of the energy of a quantum ideal gas from
the point
of view of the theory of partitions of natural numbers}.
\bjournal{Russian Math. Surveys}
\bvolume{52}
\bpages{379--386}.%
\end{barticle}%
%
\endbibitem%

\bibitem{VKantorovich}
%
\begin{barticle}[mr]
\bauthor{\bsnm{Vershik},~\bfnm{A.~M.}\binits{A.~M.}}
(\byear{2006}).
\btitle{The {K}antorovich metric: The initial history and little-known
applications}.
\bjournal{J. Math. Sci. (N. Y.)}
\bvolume{133}
\bpages{1410--1417}.
\bptnote{check year}
\end{barticle}
%
\endbibitem

\bibitem{VZ}
%
\begin{barticle}[mr]
\bauthor{\bsnm{Vershik},~\bfnm{A.}\binits{A.}} \AND
\bauthor{\bsnm{Zeitouni},~\bfnm{O.}\binits{O.}}
(\byear{1999}).
\btitle{Large deviations in the geometry of convex lattice polygons}.
\bjournal{Israel J. Math.}
\bvolume{109}
\bpages{13--27}.
\bid{doi={10.1007/BF02775023}, mr={1679585}}
\end{barticle}
%
\endbibitem

\bibitem{Widder}
%
\begin{bbook}[vtex]
\bauthor{\bsnm{Widder},~\bfnm{David~Vernon}\binits{D.~V.}}
(\byear{1941}).
\btitle{The {L}aplace {T}ransform}.
\bseries{Princeton Mathematical Series}
\bvolume{6}.
\bpublisher{Princeton Univ. Press}, \baddress{Princeton, NJ}.
\bid{mr={0005923}}
\bptnote{check year}
\end{bbook}
%
\endbibitem

\bibitem{Yeh}
%
\begin{bbook}[mr]
\bauthor{\bsnm{Yeh},~\bfnm{J.}\binits{J.}}
(\byear{1995}).
\btitle{Martingales and Stochastic Analysis}.
\bseries{Series on Multivariate Analysis}
\bvolume{1}.
\bpublisher{World Scientific}, \baddress{Singapore}.
\bid{mr={1412800}}
\end{bbook}
%
\endbibitem\

\bibitem{Z}
%
\begin{bmisc}[vtex]
\bauthor{\bsnm{Zarbaliev},~\bfnm{S.~M.}\binits{S.~M.}}
(\byear{2004}).
\bhowpublished{Limit theorems for random
convex polygonal lines. Ph.D. thesis. Moscow State Univ.,
Moscow}.
\end{bmisc}
%
\endbibitem

\end{thebibliography}
\end{document}